\documentclass[a4paper,10pt]{elsarticle}
\usepackage{amsmath}
\usepackage{amssymb}
\usepackage{amsfonts}
\usepackage{amsthm}
\usepackage{amsxtra}
\usepackage{rotating}
\usepackage{float}
\usepackage{xcolor}
\usepackage{array}
\usepackage{pifont}
\usepackage{multirow}
\usepackage{graphicx}
\usepackage[mathscr]{euscript}
\usepackage[plainpages=false,pdfpagelabels]{hyperref}
\usepackage{natbib}
\usepackage{enumitem}
\usepackage{slashbox}
\usepackage{pict2e}
\usepackage[colorinlistoftodos,prependcaption,textsize=tiny]{todonotes}
\usepackage{fixme}

\theoremstyle{plain}

\newtheorem{corollary}{Corollary}[section]
\newtheorem{lemma}{Lemma}[section]
\newtheorem{proposition}{Proposition}[section]
\theoremstyle{definition}
\newtheorem{definition}{Definition}[section]

\theoremstyle{remark}
\newtheorem{remark}{Remark}[section]

\def\8{\infty}

\newcommand{\C}{\mathbb C}
\newcommand{\R}{\mathbb R}
\newcommand{\Z}{\mathbb Z}
\newcommand{\N}{\mathbb N}
\newcommand{\Q}{\mathbb Q}

\newcommand{\half}{
	{\lower0.00ex\hbox{\raise.6ex\hbox{\the\scriptfont0 1}
			\kern-.5em\slash\kern-.1em\lower.45ex
			\hbox{\the\scriptfont0 2}}}}
\newcommand{\quarter}{
	{\lower0.00ex\hbox{\raise.6ex\hbox{\the\scriptfont0 1}
			\kern-.5em\slash\kern-.1em\lower.45ex
			\hbox{\the\scriptfont0 4}}}}
\newcommand{\tquarter}{
	{\lower0.00ex\hbox{\raise.6ex\hbox{\the\scriptfont0 3}
			\kern-.5em\slash\kern-.1em\lower.45ex
			\hbox{\the\scriptfont0 4}}}}
\newcommand{\eighth}{
	{\lower0.00ex\hbox{\raise.6ex\hbox{\the\scriptfont0 1}
			\kern-.5em\slash\kern-.1em\lower.45ex
			\hbox{\the\scriptfont0 8}}}}
\newcommand{\othird}{
	{\lower0.00ex\hbox{\raise.6ex\hbox{\the\scriptfont0 1}
			\kern-.5em\slash\kern-.1em\lower.45ex
			\hbox{\the\scriptfont0 3}}}}

\def\arg{{\rm arg}}

\begin{document}

\begin{frontmatter}
	
\title{On the Density arising from the Domain of Attraction of an Operator interpolating between Sum and Supremum: the $\alpha$-Sun Operator}

\author{N.S. Witte}
\address{School of Mathematics and Statistics, Victoria University of Wellington, PO Box 600 Wellington 6140, New Zealand}
\ead{n.s.witte@protonmail.com}

\author{P.E. Greenwood}
\address{Department of Mathematics, University of British Columbia, Vancouver BC V6T 1Z4, Canada}
\ead{pgreenw@math.ubc.ca}

\begin{abstract}
We explore the analytic properties of the density function $ h(x; \gamma,\alpha) $, $ x \in (0,\infty) $, $ \gamma > 0 $, $ 0 < \alpha < 1 $ which arises as a normed limit from the
domain of attraction problem for an operator interpolating between the
supremum and sum as applied to a sequence of i.i.d. non-negative random
variables. The parameter $ \alpha $ controls the interpolation between these two
cases, while $ \gamma $ further parametrises the type of distribution from which the
underlying random variables are drawn. It is known that in the normed
limit, for $ \alpha = 0 $ the Fr{\'e}chet density arises, whereas for $ \alpha = 1 $ the limit is
a stable random variable which has the well-known identification with a particular Fox
H-function. It is known \cite{GH_1991} that for intermediate $ \alpha $ an entirely new
distribution function appears, which is not one of the extensions to the
hypergeometric function considered to date. Here we derive series, integral
and continued fraction representations of this latter function.
\end{abstract}


\end{frontmatter}

\section{Our Motivation}\label{motivation}
\setcounter{equation}{0}

Consider a sequence of non-negative independent identically distributed random variables $ X_1, X_2, \ldots $. 
Define the sequence $ Y_0, Y_1, Y_2, \ldots $ by
\begin{equation}\label{alpha-sun}
	Y_n = \max\{Y_{n-1}, \alpha Y_{n-1}+X_n \}, \quad n \geq 1, \quad Y_0 = X_0 \in \R .
\end{equation}
For $ \alpha=0 $ we have the supremum, $ Y_n = \max\{X_0, X_1, X_2, \ldots, X_n \} $, 
while for $ \alpha = 1 $ we have the sum, $ Y_n = X_0 + \sum^{n}_{j=1} X_j \mathbb{I}_{[0,\infty)}(X_j) $. 
We see that the sequence \eqref{alpha-sun}, termed an $\alpha$-Sun sequence, interpolates between the sum and the supremum
of a sequence of i.i.d. random variables via the parameter $\alpha$.
When $ \alpha = 1 $, if a normed limit of the sequence $ Y_{n} $ exists, the limit is a stable law, 
and we say the underlying distribution is in the domain of attraction of a stable distribution. 
When $ \alpha = 0 $, if a normed limit of $ Y_{n} $ exists, the limit is a Fr{\'e}chet extreme value distribution, 
and we say the underlying distribution is in the "sup" domain of attraction of a Fr{\'e}chet law, $ \Phi_{\gamma}(x) = \exp(-x^{-\gamma})\mathbb{I}_{[0,\infty)}(x) $.
In their study of \eqref{alpha-sun}, Greenwood and Hooghiemsta [\cite{GH_1991}, Theorem 2, Eq. 2.4]
have shown that the domain of attraction of the $\alpha$-Sun operator is the same as that for the sup operator, 
i.e. it is using the same collection of underlying distributions and the same norming sequences as for the sup operator, 
that the normed $\alpha$-Sun sequences converges weakly. 
The non-degenerate limiting law satisfies a limiting distribution with a density $ h $.

There is a compelling motivation for our model \eqref{alpha-sun}.
Many decisions which we take are of the following form: 
Accept the status quo, in which case $ Y_{n} = Y_{n-1} $,
or risk losing some portion of what we have, in the process of taking some action, 
so that the value of our current state is reduced to a portion $ \alpha Y_{n-1} $, 
in which case we can increase our value by a stochastic $ X_{n} $ to obtain $ Y_{n} = \alpha Y_{n-1} + X_{n} $.
Suppose we are so lucky or wise that we can evaluate $ X_{n} $ before taking the decision whether to act or not. 
Then we can maximize the outcome by formulating our decision as per \eqref{alpha-sun}.
Historically this model first arose in descriptions of the storage of energy in a solar collection system,
see Haslett \cite{Has_1978} and the Introduction of \cite{HS_1986} for some of the early work on this problem.

In their study of this model - the $\alpha$-Sun process - Greenwood and Hooghiemstra \cite[Theorem 2, Eq. (2.4)]{GH_1991}
have shown that there exist normalising constants $a_n > 0$, $b_n \in \R$
such that $(Y_n-b_n)/a_n$ converges weakly to a non-degenerate limiting distribution which
has a density $h(x;\gamma,\alpha)$ satisfying a linear, homogeneous integral equation 
\begin{equation}\label{integral_eqn}
  h(x) = \frac{\gamma}{x} \int_{0}^{x}du\; \frac{h(u)}{(x-\alpha u)^{\gamma}} ,
\end{equation}
subject to the conditions: $ \alpha \in (0,1) $, $ \gamma \in (0,\infty) $ identifies the domain of attraction as above, 
and $ h(x) $ is a real, normalised probability density on $ x \in (0,\infty) $.
In our study we investigate the solutions $ h(x) = h(x;\gamma, \alpha) $ of this integral equation.
We will see towards the end of this paper (see Remark \ref{discontinuity}) that the $\alpha$-Sun process, 
for all $\alpha$ in $(0,1)$, behaves in a way that is much more like the supremum of the input variables than like the sum. 
There is a discontinuity of behaviour exactly when $\alpha$ becomes unity.

Our central result gives a Mellin-Barnes integral representation for $h(x)$, \eqref{MBIR_h} in Prop. \ref{MB_representations}, 
which is a standard way of defining many special functions, such as the Meijer-G and Fox-H functions.
Under the conditions $ 0 \leq \alpha < 1 $, $ \gamma>0 $, $ 0<x<\infty $ the density $ h(x;\gamma,\alpha) $ has the representation for any $ c < 1 $
\begin{equation}
	h(x) = \frac{\gamma}{x}\int_{c-i\infty}^{c+i\infty} \frac{dt}{2\pi i} x^{-\gamma t}\Gamma(1-t)(1-\alpha)^{-\gamma t}
	\prod_{j=1}^{\infty}\frac{{}_2F_1(\gamma,(j-t)\gamma;1+(j-t)\gamma;\alpha)}{{}_2F_1(\gamma,j\gamma;1+j\gamma;\alpha)}  .
\label{Mellin-Barnes_h}
\end{equation}

In a subsequent work \cite{HG_1997}, Hooghiemstra and Greenwood took up the problem of finding the remaining normed limit distributions and their domains of
attraction for the $\alpha$-Sun operator applied to sequences of i.i.d.~random sequences on all of $\R$. 
They found that there are two additional possibilities, corresponding to the two additional extreme value distributions. 
The two classes of functions and the norming constants are again the same as for the two additional extreme value distributions,
namely $ F \in D(\Psi_{\gamma})\cup D(\Lambda) $ where 
$ \Psi_{\gamma}(x) = \exp(-(-x)^{\gamma})\mathbb{I}_{(-\infty,0)}(x)+\mathbb{I}_{(0,\infty)}(x) $ and 
$ \Lambda(x) = \exp(-e^{-x}) $.
For the latter case a solution for the density was given there by Eq. (10) in Cor. 2 of \S 2,
\begin{equation}\label{gumbell}
   h(x;\alpha) = \frac{1-\alpha}{\Gamma((1-\alpha)^{-1})}\exp\left( -x-e^{-(1-\alpha)x} \right).
\end{equation}
In the former case one seeks solutions to an integral equation similar to \eqref{integral_eqn}, as given by Eq. (9) in Cor. 1,
\begin{equation}\label{integral_eqn-variant}
   h(x) = \frac{\gamma}{|x|} \int_{x/\alpha}^{x}du\; |x-\alpha u|^{\gamma}h(u) ,\quad x \in (-\infty,0) .
\end{equation}

To our knowledge no explicit solutions to either \eqref{integral_eqn} or \eqref{integral_eqn-variant} have been reported in any studies following this line of enquiry. 
Nonetheless, Hooghiemstra \& Scheffer \cite{HS_1986} have given a description in terms of the Laplace transform of the random variable with 
density $h(x)$ satisfying \eqref{integral_eqn} as a power series. 
However they did not go further than this result and one of their characterisations of this transform is incorrect.

Using different methods Schlather \cite{Sch_2001} has treated a related problem.
The Schlather paper says, at the very beginning, that the 1991 and 1997 Greenwood and Hooghiemstra
papers combine the CLT with extreme value theory. However this is not the case. 
We take up this distinction in the discussion \S\ref{discuss}.
See the more recent thesis by Anja Jan{\ss}en \cite{Jan_2010} continuing this other line of enquiry.

The study of density solutions to a variant of \eqref{integral_eqn}, specialised to $ \alpha=1 $ but possessing an additional parameter $ m \in\N $, was initiated in 
\cite{Sch_1987}, then pursued in \cite{Pak_2014} and most recently by \cite{JSW_2018}. This distinct integral equation takes the form 
\begin{equation}\label{integral_eqn_modified}
     h(x) = \frac{\gamma}{x^m} \int_{0}^{x}du\; \frac{h(u)}{(x-u)^{\gamma}} ,
\end{equation} 
and the corresponding random variables are called "generalized stable" and our case includes the classical stable case $m = 1$. 
Solutions of \eqref{integral_eqn_modified} in the special case of $ m=1-\gamma $ were found by von Wolfersdorf \cite{vWol_2007} (see also Theorem 7.2.23, pg. 265 of \cite{Brunner_2017}).
Such random variables are of interest in the case $m = 2$, $ \gamma \in (0, 1)$ which is especially investigated in Section 3 of \cite{Sch_1987} and 
Section 7 of \cite{Pak_2014}, because of its relevance to particle transport on a one-dimensional lattice. 
The first work \cite{Sch_1987} identifies the solutions in terms of Fox $H$-functions and shows that for all $m \in \N$, $\gamma \in (1-m,1)$ 
there exists a density solution to \eqref{integral_eqn_modified} having a convergent power series representation at $ x=\infty $ and a Fr{\'e}chet-like behaviour at $ x=0 $. 
The second work \cite{Pak_2014} considers invariance under length-biasing and 
shows that for all $m \in \N$, $\gamma \in (1-m,1)$ there exists a unique density solution to \eqref{integral_eqn_modified}, 
whose corresponding random variable can be represented in the case $\gamma \in (1-m,2-m)$ as a finite independent product involving the Gamma and positive stable random variables.
We have several common results with \cite{JSW_2018}, which are given later, and this latter work shows that there exists a density solution, 
which is then unique and can be expressed in terms of the Beta distribution, if and only if $m >1-\gamma$. 
These density solutions extend the class of generalised one-sided stable distributions introduced in \cite{Sch_1987} and subsequently investigated in \cite{Pak_2014}. 
This later work studies various analytical aspects of these densities, and solves some open problems about their infinite divisibility.

From the analytical viewpoint there is another thread that our study connects with, and this concerns the solutions of
integro-differential equations or more specifically differential-delay Volterra integral equations.
For this to become apparent let us consider the cumulative distribution functions
\begin{gather}
   f(x) := \int_{0}^{x}du\; h(u), \quad f(0) = 0, f(\infty) = 1, \quad \text{from the solution of \eqref{integral_eqn}} ,
\label{cdf-a}\\
   f(x) := \int_{x}^{\infty}du\; h(-u), \quad f(0) = 1, f(\infty) = 0, \quad \text{from the solution of \eqref{integral_eqn-variant}} .
\label{cdf-b}
\end{gather}
An important point to note is that henceforth we will frame our discussion of \eqref{integral_eqn-variant} for $ x\in \R_{+} $ by
defining another density $ h_{<}(x) := h(-x) $. Where there is no chance of confusion we will often drop the subscript $<$.
Thus our defining integral equation \eqref{integral_eqn-variant} will become
\begin{equation}\label{integral-eqn-mod}
      h_{<}(x) = \frac{\gamma}{x} \int_{x}^{x/\alpha}du\; (x-\alpha u)^{\gamma}h_{<}(u) ,\quad x \in (0,\infty) .
\end{equation} 
A simple exercise then recasts \eqref{integral_eqn} into the retarded version of such equations
\begin{equation}
   f'(x) - (1-\alpha)^{-\gamma}\gamma x^{-\gamma-1}f(x) = -\alpha\gamma^2 x^{-\gamma-1} \int_{0}^{1}dq\; (1-\alpha q)^{-\gamma-1} f(qx) ,
\label{DDIE-retard} 
\end{equation}
whereas \eqref{integral_eqn-variant} becomes the advanced version
\begin{equation}
  f'(x) + (1-\alpha)^{\gamma}\gamma x^{\gamma-1}f(x) = \alpha\gamma^2 x^{\gamma-1} \int_{1}^{\alpha^{-1}}dq\; (1-\alpha q)^{\gamma-1} f(qx) ,
\label{DDIE-advance} 
\end{equation}
where these are to be solved as boundary value problems, as defined above.
This pair of equations is a specialised form of a general class of integro-differential equations formulated by Iserles and Liu \cite{IL_1997},
although they did not derive solutions of the general system nor our particular case.
This class also covers the Pantograph type equations \cite{Der_1989}, \cite{BDMO_2008} and cell division and growth models \cite{HW_1989}, \cite{HW_1990}.
An early study presaging these developments is one on a stochastic absorption problem \cite{Gav_1964}.

We will treat only the solutions of \eqref{integral_eqn} or \eqref{DDIE-retard} in this work, 
and defer the solutions of \eqref{integral-eqn-mod} or \eqref{DDIE-advance} to a subsequent study.
Following this introduction we give explicit solutions to the two special cases $ \alpha=0 $ and $ \alpha=1 $ in \S \ref{special_cases}.
The density functions we find here are either well-known or well studied extensions to the hypergeometric functions,
and we include these cases because they serve as "bookends" to the novel aspects of our results.
After that we treat the general case $ 0 < \alpha < 1 $ in \S \ref{generic_case}, which can't be handled in the same way as the preceding section. Our main results are given in Prop. \ref{MB_representations} after introducing the Mellin transform of the density $H(s)$
defined by
\begin{equation}\label{MT_defn}
    H(s) := \int_{0}^{\infty} dx\; x^{s-1}h(x),
\end{equation}
and establishing a number of additional lemmas required to clarify the nature of analyticity of the transform and the convergence of 
certain infinite products.
The density function arising here is somewhat more exotic, one of a broader class of functions lying beyond the generalised hypergeometric, 
Meijer-G or Fox H-functions.
We give explicit series, integral and continued fraction representations of the specific function arising in our study.
In \S \ref{numerics} we evaluate and plot the densities for a range of 
$ \gamma \in \{\tfrac{1}{4},\tfrac{1}{2},\tfrac{3}{4},1 \} $ and $ \alpha \in \{\tfrac{1}{4},\tfrac{1}{2},\tfrac{3}{4},1\} $.
A number of implications and conclusions are taken up in the Discussion \S \ref{discuss}.

\section{Special Cases $ \alpha = 0 $ and $ \alpha = 1 $}\label{special_cases}
\setcounter{equation}{0}

\subsection{Case $ \alpha = 0 $: the Fr{\'e}chet and Weibull distributions}\label{alpha_zero}
In this case \eqref{integral_eqn} reduces to
\begin{equation}\label{integral_eqn-zero}
h_{0}(x) = \frac{\gamma}{x^{\gamma+1}} \int_{0}^{x}du\; h_{0}(u) .
\end{equation}
A simple differentiation of this yields the ordinary differential equation
\begin{equation*}
   \frac{d}{dx}h_{0} + \left[ \frac{\gamma+1}{x}-\frac{\gamma}{x^{\gamma+1}} \right]h_{0} = 0 ,
\end{equation*}
with the solution
\begin{equation*}
   h_{0} = C x^{-\gamma-1} \exp(-x^{-\gamma}) ,
\end{equation*}
and the normalisation gives $ C=\gamma $. Thus
\begin{equation}\label{h_zero}
  h_{0}(x) = \gamma x^{-\gamma-1} \exp\left[ -x^{-\gamma} \right] = \frac{d}{dx}\exp\left[ -x^{-\gamma} \right] .
\end{equation}
The solution for the case $\alpha = 0$ corresponds to a Fr{\'e}chet distribution which is the power transform of index $-1/\gamma$ of the unit exponential distribution.
See \cite{Fre_1927}, \cite{FT_1928}, \cite{Gum_1958}, \cite{Res_1987}.

It will be of interest to compute the Mellin transform of $ h_{0}(x) $ as given by \eqref{h_zero}, which we denote by $ H_{0}(s) $ for $ \Re(s) < 1+\gamma $
via the definition
\begin{equation}\label{H0defn}
	H_{0}(s) = \int^{\infty}_{0}x^{s-1} h_{0}(x) dx ,
\end{equation}
and has the evaluation
\begin{equation}\label{H_zero}
   H_0(s) = \Gamma\left( \frac{1+\gamma-s}{\gamma} \right) .
\end{equation}
Thus we can construct $ h_{0}(x) $ from the inverse Mellin transform
\begin{equation*}
	h_{0}(x) = \frac{1}{2\pi i} \int_{c-i\infty}^{c+i\infty} ds\; x^{-s}\Gamma\left( \frac{1+\gamma-s}{\gamma} \right), \quad c < 1+\gamma .
\end{equation*}

\subsection{Case $ \alpha = 1 $}\label{alpha_unity}
In this case \eqref{integral_eqn} reduces to a linear, homogeneous and singular integral equation
\begin{equation}\label{IE_unity}
   h_{1}(x) = \frac{\gamma}{x} \int_{0}^{x}du\; \frac{h_{1}(u)}{(x-u)^{\gamma}} . 
\end{equation}
It is a very well-known fact that there is a unique solution to this equation which is a positive stable distribution with index $\gamma$,
see \cite{Sch_1987},\cite{Pak_2014},\cite{JSW_2018} and the references therein.
The appearance of a positive stable distribution is expected here because we have then a renormalised sum of independent positive random variables.
\begin{proposition}[\cite{Zolotarev_1986},\cite{Samorodnitsky+Taqqu_1994},\cite{Nolan_2020}]\label{LXfm_alpha=unity}
Let $ 0 < \Re(\gamma) < 1 $. Then the solution $ h_{1}(x) $ has the integral representation
\begin{equation}\label{h_unity}
  h_{1}(x) = \frac{1}{\pi}\int_{0}^{\infty} dr\; \exp\left[ -xr-\cos(\pi\gamma)\Gamma(1-\gamma)r^{\gamma} \right]\sin\left(\frac{\pi}{\Gamma(\gamma)}r^{\gamma} \right) .
\end{equation}
\end{proposition}
\begin{proof}
Being a singular integral equation with difference kernel, \eqref{IE_unity} can be treated with a Laplace transform.
For $ \Re(p) > 0 $ let
\begin{equation*}
   L_{1}(p) := \int_{0}^{\infty} dx\; e^{-px} h_{1}(x) .
\end{equation*}
Then the Laplace transform of \eqref{IE_unity} gives us
\begin{align*}
  -\gamma^{-1}\frac{d}{dp} L_{1}(p) 
  & = \int_{0}^{\infty} dx\; \int_{0}^{x} du\; e^{-px}h_{1}(u)(x-u)^{-\gamma} ,
  \\
  & = \int_{0}^{\infty} du\;h_{1}(u) \int_{u}^{\infty} dx\; e^{-px}(x-u)^{-\gamma} ,
  \\
  & = \Gamma(1-\gamma)p^{\gamma-1} L_{1}(p) .
\end{align*}
This latter equation is easily solved for $ L_{1}(p) $ and upon the inversion of the transform yields
\begin{equation}\label{ILT_unity}
   h_{1}(x) = \frac{C}{2\pi} \int_{c-i\infty}^{c+i\infty} dp\; e^{px-\Gamma(1-\gamma)p^{\gamma}}, \quad c>0 , 
\end{equation}
where $C$ is a normalisation constant.
The integrand of the inversion formula has no non-analytic features other than a branch cut in the $p$-plane on $ [0,-\infty) $ and it is possible to deform the vertical
contour to a Hankel loop $ (-\infty, 0, -\infty) $ either side of the cut.
Subsequently, in the second proof of Prop. \ref{H_unity_Prop}, we will see that $ C=1 $.
This then gives the stated result, \eqref{h_unity}. 
\end{proof}

\begin{remark}
When $ \alpha=1 $ we are inevitably led to the additional constraint of $ \Re(\gamma) < 1 $.
This is clear from the strength of the singularity at the end-point of the integration range in \eqref{IE_unity} and Gamma function $ \Gamma(1-\gamma) $ in \eqref{ILT_unity}.
\end{remark}

\begin{proposition}[Eq. (3.18) of \cite{Sch_1987}, Eq. (14.31) (with $\rho=1$) of \cite{Sat_1999}]
For $ \Re(\gamma) < 1 $ the density $ h_{1}(x) $ has a convergent large-$x$ expansion
\begin{equation}\label{h_unity_expand}
   h_{1}(x) = \frac{1}{\pi x} \sum_{n=1}^{\infty} (-1)^{n-1} \sin(\pi n\gamma) \frac{\Gamma(1+n\gamma)}{n!}\left( \frac{\Gamma(1-\gamma)}{x^{\gamma}} \right)^n .
\end{equation}
\end{proposition}
\begin{proof}
Rewriting the integrand of \eqref{h_unity} in terms of exponentials instead of trigonometric functions and then expanding 
$ \exp\left(-e^{\pm\pi i \gamma}\Gamma(1-\gamma)r^{\gamma} \right) $ in a Taylor series about $ r^{\gamma}=0 $, we observe the resulting integrals can be done
after recognising the sum and integral are uniformly and absolutely convergent.  
\end{proof}

\begin{remark}
Formula \eqref{h_unity_expand} can be simplified when $ \gamma \in \Q $.
If $ \gamma $ is rational then Gauss's multiplication formula for the Gamma function \cite[Eq. (5.5.6)]{DLMF} can be employed to write the sum in a hypergeometric form.
For example one has
\begin{align}\label{alpha-unity_gamma-rational}
   h_{1}(x;\tfrac{1}{2}) & = \frac{1}{2x^{3/2}}e^{-\pi/4x} ,
   \\
   h_{1}(x;\tfrac{1}{3}) & = \frac{\Gamma \left(\frac{2}{3}\right)}{\sqrt[3]{3} x^{4/3}} 
                             \text{Ai}\left( 3^{-1/3}\Gamma(\tfrac{2}{3})x^{-1/3} \right) ,
   \\
   h_{1}(x;\tfrac{2}{3}) & = 
   \frac{4\pi e^{-2\Gamma\left(\frac{4}{3}\right)^3/x^2}}{3^{5/6} \Gamma\left(-\frac{1}{3}\right) x^{7/3}}
   \\ \nonumber
   & \times
   \left[ 3^{2/3} x^{2/3} \text{Ai}'\left( 3^{2/3} \Gamma\left(\tfrac{4}{3}\right)^2x^{-4/3} \right)
   	 -3 \Gamma\left(\tfrac{4}{3}\right) \text{Ai}\left( 3^{2/3} \Gamma\left(\tfrac{4}{3}\right)^2x^{-4/3} \right) \right] ,
\end{align}
using the standard definitions of Airy function \cite[Chapter 9]{DLMF}.
\end{remark}

The following result for general $ \gamma $ appears to be simplest identification amongst the extensions of hypergeometric function families.
\begin{proposition}[Eq. (3.16) of \cite{Sch_1987}, \S 3.1.3, \cite{JSW_2018}]
The density $ h_{1}(x;\gamma) $ is the Fox $H$-function (Fig. \ref{fig:alpha=unity})
\begin{equation}\label{hFox_unity}
   h_{1}(x;\gamma) = \gamma^{-1}\left( \Gamma(1-\gamma) \right)^{-1/\gamma} 
   {\rm H}^{0,1}_{1,1} \left( \frac{x}{\left( \Gamma(1-\gamma) \right)^{1/\gamma}};
   \begin{array}{c} (1-\gamma^{-1},\gamma^{-1}) \\ (0,1) \end{array} \right) ,
\end{equation}
as per the definition of \cite{KL_2019}[Eq. (51), pg. 316, Chap. 7].
\end{proposition}
\begin{proof}
This follows from the Mellin-Barnes integral representation of the Fox $H$-function, see \eqref{H_unity}.
This also agrees with the formulae in \S 3.1.3 of \cite{JSW_2018}, with $ m=n=1 $ and the scale factor $ c=\left( \gamma\Gamma(1-\gamma) \right)^{-1/\gamma} $.
\end{proof}

\begin{figure}[H]
    \includegraphics[]{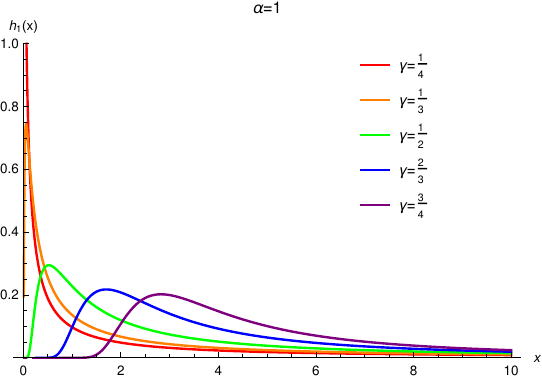}
	\caption{Densities $ h_1(x;\gamma) $ versus $ x $ for $ \gamma = \{ \frac{1}{4}, \frac{1}{3}, \frac{1}{2}, \frac{2}{3}, \frac{3}{4} \}$.}
	\label{fig:alpha=unity}       
\end{figure}

\section{General $ \alpha $ Case}\label{generic_case}
\setcounter{equation}{0}

The kernel in \eqref{integral_eqn} is not separable nor of convolution type unless $ \alpha=1 $ so we need to employ more powerful methods in the generic $ \alpha $ case, than those utilised in \S \ref{special_cases}.
In the general case we recall the Mellin transform of the density $ h(x) $ is defined by
\begin{equation}
    H(s) := \int_{0}^{\infty} dx\; x^{s-1}h(x),
\end{equation}
for a suitable vertical strip in the $s$-plane, to be clarified below, but containing the line $ \Re(s) = 1 $.
This latter condition constitutes our normalisation $ H(1)=1 $.

\begin{proposition}
For $ 0 < \alpha < 1 $, $ \gamma > 0 $ and $ \Re(s) < 1+\gamma $ the retarded Mellin transform $ H(s;\gamma,\alpha) $ satisfies the linear, homogeneous functional equation
\begin{equation}\label{MT_functional}
	H(s) = \frac{\gamma}{1+\gamma-s} {}_2F_1(\gamma,1+\gamma-s;2+\gamma-s;\alpha) H(s-\gamma) ,
\end{equation}
where $ {}_2F_1(a,b;c;z) $ is the Gauss hypergeometric function \cite[\S 15]{DLMF}. 
\end{proposition}
\begin{proof}
Taking the Mellin transform of both sides of \eqref{integral_eqn} we compute
\begin{align*}
  \int_{0}^{\infty} dx\; x^{s-1}h(x) 
  & = \gamma \int_{0}^{\infty} dx\; x^{s-2} \int_{0}^{x}du\; \frac{h(u)}{(x-\alpha u)^{\gamma}} ,
  \\
  & = \gamma \int_{0}^{\infty} du\; h(u) \int_{u}^{\infty} dx\; \frac{x^{s-2}}{(x-\alpha u)^{\gamma}} ,
  \\
  & = \gamma \int_{0}^{\infty} du\; h(u) u^{s-1-\gamma} \int_{1}^{\infty} dx\; \frac{x^{s-2}}{(x-\alpha)^{\gamma}} .
\end{align*} 
The nested integral exists only for $ \Re(s) < 1+\gamma $ and the integrand permits an expansion about $ \alpha=0 $, 
which can be integrated term-by-term to yield
\begin{equation*}
  \int_{1}^{\infty} dx\; \frac{x^{s-2}}{(x-\alpha)^{\gamma}} =
  \frac{1}{1+\gamma-s}  \sum^{\infty}_{k=0} \frac{(\gamma)_{k}(1+\gamma-s)_{k}}{k!(2+\gamma-s)_{k}}\alpha^k .
\end{equation*}
This allows us to make the identification with the stated Gauss hypergeometric function and \eqref{MT_functional}.

From its definition \eqref{MT_defn} $ H(s) $ exists for $ s $ in the strip $ a < \Re(s) < b $, for some $ a<b $. 
Due to the existence of the norm we have $ a < 1 < b $.
In addition $ H(s-\gamma) $ exists for $ a < \Re(s)-\gamma < b $, or $ a+\gamma < \Re(s) < b+\gamma $.
There must be a common overlap of these two intervals so we conclude $ a+\gamma < b $.
Exploiting all of the above inequalities we find that $ b < 1+\gamma $ and $ 1-\gamma < a $ and so
in summary we deduce $ 1-\gamma < a < 1 < b < 1+\gamma $.  
\end{proof}

\subsection{On the nature of the ${}_2F_1(\gamma,1+\gamma-s;2+\gamma-s;\alpha)$}
Firstly with regard to the analytical character of this hypergeometric function with respect to $ s $ we can make the following immediate observations:
\begin{enumerate}[label=(\alph*)]
	\item
	For $ 0 < |\alpha| < 1 $ the ${}_2F_1$ is convergent for all $ \Re(s) < 1+\gamma $.
	The ${}_2F_1$ has simple poles on the right at $ s =1+\gamma+k $ for $ k \in \Z_{\geq 0} $.
	\item
	For $ |\alpha| = 1 $ we require in addition that $ \Re(\gamma) < 1 $.
	The poles are the same as for $ |\alpha| < 1 $,
	\item
	Our ${}_2F_1$ function is one of a pair of solutions to the hypergeometric differential equation about $ \alpha=0 $ 
	that is bounded as $ s \to -\infty $ (by $ (1-\alpha)^{-\gamma} $). 
	This distinguishes itself from the other member of the pair which has the simple algebraic form
	$ \alpha^{s-\gamma-1} $, and which diverges as $ s \to -\infty $.
\end{enumerate}

This hypergeometric function can be linearly mapped into other forms via the Kummer or linear fractional transformations of the $ \alpha $-plane,
and we record three versions \cite[15.8.1]{DLMF} of these
\begin{align*}
   {}_2F_1(\gamma,1+\gamma-s;2+\gamma-s;\alpha) 
   = (1-\alpha)^{-\gamma} & {}_2F_1(\gamma,1;2+\gamma-s;\tfrac{\alpha}{\alpha-1}) ,
   \\
   = (1-\alpha)^{s-\gamma-1} & {}_2F_1(2-s,1+\gamma-s;2+\gamma-s;\tfrac{\alpha}{\alpha-1}) ,
   \\
   = (1-\alpha)^{1-\gamma} & {}_2F_1(2-s,1;2+\gamma-s;\alpha) .
\end{align*}

The hypergeometric function is a special case of the most general such function, 
because $c=b+1$, and the equation above contains only one general argument and two general parameters.
Thus it can be identified in a number of ways.
In the first instance it is an incomplete Beta function, see \cite[8.17.7]{DLMF},
\begin{equation*}
   {}_2F_1(\gamma,1+\gamma-s;2+\gamma-s;\alpha) = (1+\gamma-s)\alpha^{s-\gamma-1}B_{\alpha}(1+\gamma-s,1-\gamma) . 
\end{equation*}
In the second instance it is a generalisation of the Hurwitz-Lerch zeta function, namely the higher-order Apostol-Lerch function defined by the integral
\cite{LJW_2012a}, \cite{Luo_2006}, \cite{HQU_2015}
\begin{equation}
   \zeta(z,\sigma,v,\lambda) = \frac{1}{\Gamma(\sigma)} \int_{0}^{\infty}dx\; \frac{x^{\sigma-1}e^{-vx}}{\left( 1-z e^{-x}\right)^{\lambda}} ,
\label{apostol-lerch}
\end{equation}
for $ |z| \leq 1 $, $ \Re(v) > 0 $, $ \Re(\sigma) > 1 $ and $ \Re(\lambda) >  0 $. Thus
\begin{equation}
    {}_2F_1(\gamma,1+\gamma-s;2+\gamma-s;\alpha) = (1+\gamma-s) \zeta(\alpha,1,1+\gamma-s,\gamma) .
\label{Fj_apostol-lerch}
\end{equation}

A useful alternative integral representation of the $ {}_2F_1 $ is
\begin{equation}\label{F_integral}
   {}_2F_1(\gamma,1;2+\gamma-s;\tfrac{\alpha}{\alpha-1}) = 1-\alpha\gamma(1-\alpha)^{\gamma} \int_{0}^{1} dt\; t^{1+\gamma-s}(1-\alpha t)^{-\gamma-1} .
\end{equation}

\subsection{Inequalities}\label{inequalities}
We will require bounds on the density in our subsequent analysis, which can be simply derived.
\begin{proposition}
For $ 0 < x < \infty $, $ 0 < \alpha < 1 $ and $ \gamma > 0 $ the retarded density $ h(x;\gamma,\alpha) $ is a real, 
positive density and satisfies the bounds
\begin{equation}\label{h_bound}
  \gamma x^{-\gamma-1} \exp\left[ -(1-\alpha)^{-\gamma}x^{-\gamma} \right] \leq h(x) \leq \gamma (1-\alpha)^{-\gamma}x^{-\gamma-1} \exp\left[ -x^{-\gamma} \right] .
\end{equation} 
\end{proposition}
\begin{proof}
We start by noting the simple inequality $ x^{-\gamma} \leq |x-\alpha u|^{-\gamma} \leq (1-\alpha)^{-\gamma}x^{-\gamma} $ for $ 0 \leq u \leq x $.
Furthermore define the cumulative distribution function $ f(x) := \int_{0}^{x}du\; h(u) < 1 $ in the standard way.
Employing this inequality in \eqref{integral_eqn} we have
\begin{equation*}
   \gamma x^{-\gamma-1}f(x) \leq h(x)=f'(x) \leq \gamma(1-\alpha)^{-\gamma}x^{-\gamma-1}f(x) .
\end{equation*} 
Dividing by $ f(x) $ and integrating the bounds and the logarithmic derivative from the reference point $ x_{0}=\infty $ ($ f(\infty)=1 $) to $ x<\infty $ without crossing any standard placement of the branch cut, e.g. $ (-\infty,0] $, we deduce a reversed inequality
\begin{equation*}
   \exp\left[ -(1-\alpha)^{-\gamma}x^{-\gamma} \right] \leq f(x) \leq \exp\left[ -x^{-\gamma} \right] .
\end{equation*} 
Combining these two inequalities we conclude \eqref{h_bound}.   
\end{proof}

\begin{corollary}
Let $ \Re(s) = \sigma $. For $ \sigma < 1+\gamma $ the Mellin transform $ H(s;\gamma,\alpha) $ satisfies the bound
\begin{equation}\label{H_bound}
|H(s)| \leq (1-\alpha)^{-\gamma}\Gamma\left( \frac{1+\gamma-\sigma}{\gamma} \right) .
\end{equation}
\end{corollary}
\begin{proof}
This follows immediately from \eqref{h_bound}. 
It should be noted that the stronger result $ |H(s)| \leq 1 $ holds when $ \Re(s)=1 $.
\end{proof}

If one were to define a new transform $ K(s) $ by
\begin{equation}\label{K_defn}
   K(s) := \frac{(1-\alpha)^s}{\Gamma(1+\gamma^{-1}-\gamma^{-1}s)} H(s) ,
\end{equation}
then the functional equation becomes
\begin{equation}\label{K_functional_eqn}
    K(s) = {}_2F_1(\gamma,1;2+\gamma-s;\tfrac{\alpha}{\alpha-1}) K(s-\gamma) .
\end{equation}
This, or slight variants of it, will turn out to have better analytic properties with respect to $ s $.

\subsection{Consistency with $ \alpha = 0 $}\label{alpha=0_check}
At $ \alpha = 0 $ the retarded functional equation \eqref{MT_functional} reduces to
\begin{equation}\label{FE_zero}
   H_{0}(s) = \frac{\gamma}{1+\gamma-s} H_{0}(s-\gamma) .
\end{equation}

\begin{proposition}\label{Prop_zero}
The solution of \eqref{FE_zero} is 
\begin{equation}\label{H_zero-soln}
   H_0(s) = \Gamma\left( \frac{1+\gamma-s}{\gamma} \right) .
\end{equation}
\end{proposition}
\begin{proof}
Equation \eqref{FE_zero} is a linear, homogeneous first order recurrence relation with the particular solution \eqref{H_zero} or \eqref{H_zero-soln}
which also satisfies the boundary condition $ H(1)=1 $.
It is also the only solution because the only permissible $\gamma$-periodic function $ P(s) $, satisfying $ P(s)=P(s-\gamma) $, 
and making $ H_{0}(s)P(s) $ another solution, 
is the constant function. 
This crucial fact follows because $ |H(s)| $ must be bounded (in fact by unity according to our hypothesis stated following \eqref{MT_defn}) as $ \Im(s) \to \pm \infty $ along the line $ \Re(s)=1 $.
\end{proof}

We can also verify \eqref{FE_zero} is satisfied by \eqref{H_zero} from the direct computation of its Mellin transform as given by \eqref{h_zero}.

\subsection{Consistency with $ \alpha = 1 $}\label{alpha=1_check}
For $ \alpha = 1 $ the Gauss hypergeometric function evaluates to
\begin{equation*}
   {}_2F_1(\gamma,1+\gamma-s;2+\gamma-s;1) = \frac{\Gamma(1-\gamma)\Gamma(2+\gamma-s)}{\Gamma(2-s)} ,
\end{equation*}
and the functional equation becomes 
\begin{equation}\label{FE_unity}
  H_{1}(s) = \frac{\gamma\Gamma(1-\gamma)\Gamma(1+\gamma-s)}{\Gamma(2-s)} H_{1}(s-\gamma) .
\end{equation}

The following solution for $H_{1}(s)$ is a well-known fact about the positive stable laws, 
which can be easily recovered from the Laplace transform $L_{1}(p)$ given in the proof of Prop. \ref{LXfm_alpha=unity} as we illustrate in our second proof.
However our first proof will utilise arguments relying only on the above recurrence relation.
\begin{proposition}[see Eq. (11), \S 2.1 of \cite{JSW_2018}]\label{H_unity_Prop}
The solution of \eqref{FE_unity} is
\begin{equation}\label{H_unity}
   H_{1}(s) = \left( \Gamma(1-\gamma)\right)^{(s-1)/\gamma} \frac{\Gamma\left(\frac{\displaystyle 1+\gamma-s}{\displaystyle \gamma}\right)}{\Gamma(2-s)} .
\end{equation}
\end{proposition}
\begin{proof}[First Proof]
Define $ t=s/\gamma $ and $ P_{1}(t) $ by
\begin{equation*}
    H_{1}(s) = \left( \Gamma(1-\gamma)\right)^{t}\sin(\pi\gamma t)\Gamma(\gamma t)\Gamma(\gamma^{-1}-t)P_{1}(t) .
\end{equation*}
Then \eqref{FE_unity} translates as $ P_{1}(t+1)=P_{1}(t) $ and so $ P_{1} $ is a $1$-periodic function.
However $ P_{1}(t) $ is, in fact a constant, due to identical arguments given in the proof of Proposition \ref{Prop_zero}.
This constant value is
\begin{equation*}
   P_{1} = \frac{1}{\pi\gamma \left( \Gamma(1-\gamma) \right)^{1/\gamma}} ,
\end{equation*}
and the stated result follows.
This result, \eqref{H_unity}, can also be found from a specialisation of Eq. (11), \S 2.1 of \cite{JSW_2018}.
\end{proof}
\begin{proof}[Second Proof]
We can directly compute the Mellin transform of \eqref{h_unity} and verify \eqref{H_unity}.
Taking this transform we compute
\begin{equation*}
   H_{1}(s) = \frac{C\,\Gamma(s)}{\pi} \int_{0}^{\infty} dr\; r^{-s} \sin\left( \frac{\pi}{\Gamma(\gamma)}r^{\gamma} \right) e^{-\cos(\pi\gamma)\Gamma(1-\gamma)r^{\gamma}} .
\end{equation*}
under the condition $ \Re(s) < 1+\gamma $ and the more restrictive $ \Re(\gamma) < \frac{1}{2} $, ($ \cos(\pi\gamma) > 0 $).
Under an internal change of variable this integral becomes a known one, \cite{GR_1980}[pg. 490, 3.944, \# 5], with an evaluation leading to \eqref{H_unity}.
In addition we see that $ C=1 $.
\end{proof}

\subsection{The general solution}

Treating the functional equation \eqref{MT_functional} as a recurrence relation and iterating the recurrence to the left $k$ times we deduce
\begin{equation}\label{H_global}
   H(s-k\gamma) = \prod_{l=1}^{k} \left(\frac{1+l\gamma-s}{\gamma}\right)\frac{1}{{}_2F_1(\gamma,1+l\gamma-s;2+l\gamma-s;\alpha)} H(s) ,
\end{equation}
and thus setting $ s=1 $, $ k\geq 0 $
\begin{equation}\label{H_k-values}
  H_{k} := H(1-k\gamma) = k! \prod_{l=1}^{k} \frac{1}{{}_2F_1(\gamma,l\gamma;1+l\gamma;\alpha)} .
\end{equation}
This is meaningful because the denominators never vanish and $ H(s) $ is analytic for $ \Re(s) < 1+\gamma $ as we demonstrate in the following Lemma.
We observe that the above relation defines a sequence of positive integer moments bounded from
below by $k!$ and hence satisfying Carleman’s condition. 
Consequently the distribution of the random variable with Mellin transform $H(1 + \gamma - s\gamma)/H(1)$ is determined by its integer moments.

\begin{definition}\label{F-defn}
Let $ F_{j} := {}_2F_1(\gamma,j\gamma;1+j\gamma;\alpha) $ for $ j \in \Z_{\geq 0} $. 
\end{definition}
\begin{lemma}
Let $ 0 < \alpha < 1 $ and $ \gamma > 0 $.
The hypergeometric factors $ F_j $ satisfy
\begin{enumerate}
	\item $ F_{0} = 1 $,
	\item are monotonically increasing, $ F_{j+1} > F_{j} $, and
	\item are bounded above, $ F_{j} < (1-\alpha)^{-\gamma} $.
\end{enumerate}
\end{lemma}
\begin{proof}
The factors have the integral representation
\begin{equation}\label{F_hyper}
   F_{j} = (1-\alpha)^{-\gamma} - \alpha\gamma \int_{0}^{1} du\; u^{j\gamma}(1-\alpha u)^{-\gamma-1} .
\end{equation}
From this we observe that 
\begin{align*}
   F_{j+1}-F_{j} = \alpha\gamma \int_{0}^{1} du\; u^{j\gamma}(1-u^{\gamma})(1-\alpha u)^{-\gamma-1} > 0.
\end{align*}
In addition $ \alpha\gamma \int_{0}^{1} du\; u^{j\gamma}(1-\alpha u)^{-\gamma-1} > 0 $ so that $ F_{j} < (1-\alpha)^{-\gamma} $.
Note that $ F_{\infty} = {}_2F_1(\gamma,-;-;\alpha) = (1-\alpha)^{-\gamma} $.
\end{proof}

\begin{remark}
In \cite{HS_1986} an identical variable is studied, $ B^{-1}_j = F_j $.
It is asserted that $ (j+1)/B_{j+1} $ has a finite limit as $ j\to\infty $, i.e. that $ B_{j+1} $ grows as $ {\rm O}(j) $ 
which contradicts our result above. The contradiction arises because of the assumption that the integral in the third line of the displayed equation on pg. 127
vanishes in the limit.
\end{remark}

In contrast, iterating the recurrence relation to the right one encounters the simple poles at $ s = 1+k\gamma $, 
as exemplified by the formula
\begin{equation}
   H(s+k\gamma) = \gamma^{k}\prod_{l=0}^{k-1} \frac{1}{1-l\gamma-s}{}_2F_1(\gamma,1-l\gamma-s;2-l\gamma-s;\alpha) \times H(s) .
\end{equation}

\begin{proposition}\label{large-k-asymptotics}
Let $ 0 \leq \alpha < 1 $, $ \gamma > 0 $, $ N \in \N $ and $ s \in \C, \Re(s) < 1+\gamma $. 
The consecutive product of hypergeometric functions has the leading order behaviour as $ N \to \infty $
\begin{equation}\label{2F1-product_asymptotic:a}
	\prod_{k=1}^{N} {}_2F_1(\gamma,1;2+k\gamma-s;\tfrac{\alpha}{\alpha-1}) \mathop{\sim}\limits_{N \to \infty} C \left[ N+1+\frac{2-s}{\gamma} \right]^{-\alpha/(1-\alpha)} ,
\end{equation}
where $ C $ is a constant independent of $ N $.
An alternative form of this is 
\begin{equation}\label{2F1-product_asymptotic:b}
   \prod_{k=1}^{N} {}_2F_1(\gamma,k\gamma;1+k\gamma;\alpha) \mathop{\sim}\limits_{N \to \infty} 
   C (1-\alpha)^{-N\gamma}\left[ N+1+\frac{1}{\gamma} \right]^{-\alpha/(1-\alpha)} .
\end{equation}
\end{proposition}
\begin{proof}
Firstly we write
\begin{equation}
   (1-\alpha)^{\gamma}{}_2F_1(\gamma,1-s+k\gamma;2-s+k\gamma;\alpha) = {}_2F_1(\gamma,1;2+k\gamma-s;\tfrac{\alpha}{\alpha-1}) = 1+A_{k}+B_{k} ,
\label{2F1split}
\end{equation}
where
\begin{equation*}
  A_{k} = -\frac{\alpha}{1-\alpha} \frac{\gamma}{2-s+k\gamma} ,
  B_{k} = \frac{\alpha^2\gamma(\gamma+1)(1-\alpha)^{\gamma}}{2-s+k\gamma}\int_{0}^{1} du u^{2-s+k\gamma}(1-\alpha u)^{-\gamma-2} .
\end{equation*}
This decomposition can be seen by subtracting the first two terms on the right-hand side of \eqref{2F1split} from the middle expression, 
then using the integral representation \eqref{F_integral} for the hypergeometric function and integration by parts. All terms cancel except for one - the term $ B_{k} $. 
This method is valid because $ \Re(s)< 1+\gamma $ assures that $ 2-\Re(s)+k\gamma >0 $ for $ k\geq 1 $.
We will also require $ 1+A_{k} >0 $ and this entails the slightly stronger condition $ 2-\Re(s)+k\gamma > \gamma\frac{\alpha}{1-\alpha} $.
Splitting the forgoing sum in the following way
\begin{equation}\label{1st_error}
  \log \left( 1+A_{k}+B_{k} \right) = \log \left( 1+A_{k} \right) +  \log \left( 1+\frac{B_{k}}{1+A_{k}} \right) ,
\end{equation}
and employing the elementary inequality $ |\log \left( 1+\frac{B_{k}}{1+A_{k}} \right)| \leq \frac{|B_{k}|}{|1+A_{k}|} $ along with the bound
\begin{equation*}
   |B_{k}| \leq \frac{\alpha^2}{(1-\alpha)^2} \frac{\gamma(\gamma+1)}{(2-s+k\gamma)(3-s+k\gamma)} ,
\end{equation*} 
we see that the second term in \eqref{1st_error} is of $ {\rm O}(k^{-2}) $ as $ k\to \infty $.
Thus in the sum over $ k $ it contributes a term of only $ {\rm O}(1) $ as $ N \to \infty $.

Now focusing on the first term of \eqref{1st_error} we see that this is bracketed by another elementary inequality
$ A_{k}-\frac{1}{2}\frac{A_{k}^2}{1+A_{k}} \leq \log \left( 1+A_{k} \right) \leq A_{k}-\frac{1}{2}A_{k}^2 $ (as $ A_{k}<0 $).
Again the second terms of the lower and upper bounds are of $ {\rm O}(k^{-2}) $ as $ k\to \infty $ 
and contribute a term of only $ {\rm O}(1) $ as $ N \to \infty $. 
Collecting all of these contributions in the $k$-sum we have, as $ N \to \infty $,
\begin{align*}
  \sum_{k=1}^{N} \log {}_2F_1(\gamma,1;2+k\gamma-s;\tfrac{\alpha}{\alpha-1})
  & = -\frac{\alpha}{1-\alpha}\sum_{k=1}^{N} \frac{\gamma}{2-s+k\gamma} + {\rm O}(1) ,
\\
  & = -\frac{\alpha}{1-\alpha} \psi(N+1+(2-s)/\gamma) + {\rm O}(1) ,
\\
  & = -\frac{\alpha}{1-\alpha} \log \left[ N+1+(2-s)/\gamma \right] + {\rm O}(1) .
\end{align*}
Thus \eqref{2F1-product_asymptotic:a} follows.
\end{proof}

Let us begin by recalling the previous definition \eqref{K_defn}
\begin{equation*}
   K(s) := \frac{(1-\alpha)^s}{\Gamma(1+\gamma^{-1}-\gamma^{-1}s)} H(s) .
\end{equation*}
For $ k \in \Z_{\geq 0} $, $ s_{k}=1-k\gamma $ we have a denumerable sequence of $ K(s_k) $ values and we know that $ K(s) $ is bounded in the finite $s$-plane.
Clearly $ K(s) $ is entire, and the question we pose here is that of interpolation of an entire function, i.e. reconstructing it from this sequence of values on $\Z_{\geq 0}$.
This question was addressed in 1884 with the Mittag-Leffler theorem \cite{Mit_1884}, \cite{Dav_1975}, \cite{Con_1978}, \cite{Rem_1998},
by the Pringsheim interpolation formula \cite{Pri_1932}, and by Guichard's theorem \cite{Gui_1884}.
However use of these results involves restrictions which are not satisfied by $ K(s) $
\footnote{The fact that $ K(s) $ can be reconstructed from its integer values as a direct consequence of Carleman’s condition on the moment problem does not help us construct a function that interpolates these values at $ s=s_{k} $ for $ s $ in the finite complex plane. Thus Carleman's condition is not relevant for this task.
}. 

A more useful function turns out to be
\begin{equation}\label{G_defn}
   G(t) := \frac{1}{\Gamma\left(\dfrac{1+\gamma-s}{\gamma}\right)}H(s) ,\quad t := \frac{1-s}{\gamma} ,
\end{equation}
which eliminates the simple poles at $ s = 1+k\gamma $, $ k \in \N $, so that in the finite-$t$ plane
\begin{equation}\label{G_bound}
   |G(t)| \leq (1-\alpha)^{-\gamma} \left[ \cosh(\pi\Im(t)) \right]^{1/2} .
\end{equation}
However, this is at the expense of a growth $ \exp(\pi |\Im(t)|/2) $ as $ \Im(t)\to \pm\infty $. 
At the prescribed values $ s=1-k\gamma $, $ k \in \Z_{\geq 0} $, we deduce 
\begin{equation}\label{G-values}
   G_{k} := G(t=k) = \frac{1}{\prod_{l=1}^{k} {}_2F_1(\gamma,l\gamma;1+l\gamma;\alpha)} ,
\end{equation}
which have the growth properties
\begin{equation}\label{G_asymptote}
   G_{k} \mathop{\sim}\limits_{k \to \infty} C (1-\alpha)^{k\gamma} \left[ k+1+\frac{1}{\gamma} \right]^{\alpha/(1-\alpha)} .
\end{equation}

\begin{definition}
Recall $ F_{j} := {}_2F_1(\gamma,j\gamma;1+j\gamma;\alpha) $ for $ j \in \N $, and let $ t=(1-s)/\gamma $. 
Define the analytic generating function $g(x;\gamma,\alpha)$
\begin{equation}\label{G_interpolate}
   g(x) := \sum_{l=0}^{\infty} \frac{(-x)^l}{l!}G_{l} = \sum_{l=0}^{\infty} \frac{(-x)^l}{l!} \frac{1}{\prod_{j=0}^{l} F_j} .
\end{equation}
\end{definition}

\begin{remark}
In \cite{HS_1986} a function is defined $\omega_{\alpha}(x)$, along with $\rho=\gamma$, which is identical to $g(x;\gamma,\alpha)$.
It is in fact an entire function of $x$ contrary to the assertion made there as we noted in a previous remark.
\end{remark}

\begin{proposition}
$ g(x) $ is an entire function of $ x $ for $ \gamma > 0 $ and $ 0 < \alpha < 1 $ due to \eqref{G_asymptote}.
The order of $ g(x) $ is unity and the type $ (1-\alpha)^{\gamma} $ \cite{Levin_1996}.
\end{proposition}

Because of the bounded growth \eqref{G_bound} and consequent analyticity of $ G(t) $ we can use the Ramanujan interpolation formula,
also known as the Ramanujan master theorem, given in the first of his quarterly reports of 1913, 
and published in \cite[\S1.2, pg. 298, Ramanujan Notebooks, Part I]{Ber_1985}
\begin{equation}\label{ramanujan}
   \Gamma(t)G(-t) = \int_{0}^{\infty} dx\; x^{t-1}g(x) = \mathcal{M}[g(x);t] .
\end{equation}
See \cite[\S 10.12, pg. 550 and Ex. 34(b)]{AAR_1999} and \cite[Chap. XI]{Har_1940} for proofs of this relation under the same hypotheses of Carlson's theorem, and \cite[App. B]{Can_2017} for recent developments with these methods.
Here we recall some details of Hardy's theorem and proof, which is framed in terms of the quantity
\begin{equation}
  \phi(u) := \frac{G(u)}{\Gamma(1+u)} ,
\end{equation}
and the half-plane $ \mathcal{H}(\delta) := \{\Re(u)>-\delta, 0<\delta<1 \} $.

\begin{proposition}[Hardy 1940, Chap. XI]\label{Hardy_RMT}
Let
\begin{enumerate}
	\item $ \phi(u) $ be regular in $ \mathcal{H}(\delta) $,
	\item $ |\phi(u)| < C \exp\left[ P\Re(u)+A|\Im(u)| \right] $ in $ \mathcal{H}(\delta) $, and
	\item $ A<\pi $.
\end{enumerate}
Then \eqref{ramanujan} is valid.
\end{proposition}
In our case we have \eqref{G_bound} and 
\begin{equation}
   |\phi(u)| \leq (1-\alpha)^{-\gamma} \frac{\cosh(\pi\Im(u))}{\Gamma(1+\Re(u))} ,
\end{equation}
so $ A=\pi $, and we need to extend the result of Hardy.
\begin{corollary}[\cite{CQ_2012}, Lemma 3.1]
Prop. \ref{Hardy_RMT} extends to $ A=\pi $, in particular the integrand of the inverse Mellin transform of \eqref{ramanujan} has algebraic decay
\begin{equation}\label{G_M-B-inversion}
  \left| \Gamma(t)G(-t) \right| \leq (1-\alpha)^{-\gamma}\Gamma(1-\Re(t)) \left| \Im(t) \right|^{2\Re(t)-1} ,
\end{equation}
as $ \left| \Im(t) \right| \to \infty $.
This means that $ G(-t) \in Q_{\infty}(2,\tfrac{1}{2}) $, one of the classes defined in \cite{CQ_2012}.
Consequently the inverse Mellin transform converges absolutely and uniformly for $ \Re(t)<0 $.
\end{corollary}
\begin{proof}
From \eqref{H_bound} we have 
\begin{equation}
   \left| G(-t) \right| \leq (1-\alpha)^{-\gamma} \frac{\Gamma(1-\Re(t))}{\left|\Gamma(1-t)\right|} ,
\end{equation}
and from Eq. 2.4.4 of \S 2.4 of \cite{PK_2001} we have
\begin{align*}
   \log \left| \Gamma(t) \right| \mathop{\to}\limits_{\left| \Im(t) \right| \to \infty} & -\tfrac{1}{2}\pi\left| \Im(t) \right| + (\Re(t)-\tfrac{1}{2})\log\left| \Im(t) \right| ,
\\
   \log \left| \Gamma(1-t) \right| \mathop{\to}\limits_{\left| \Im(t) \right| \to \infty} & -\tfrac{1}{2}\pi\left| \Im(t) \right| + (\tfrac{1}{2}-\Re(t))\log\left| \Im(t) \right| .
\end{align*}
Thus the exponential decay of the integrand vanishes and one is left with the algebraic decay of \eqref{G_M-B-inversion}.
\end{proof}

A direct relation between $ h(x) $ and $ g(x) $ is expressed in the following lemma.
\begin{lemma}
Let $ g(x) $ be defined by \eqref{G_interpolate}.
Then the density $ h(x;\gamma,\alpha) $ is given by the Hankel transform of order zero of $ g(\cdot) $
\begin{equation}\label{h_generic}
   h(x) = \gamma^2 x^{-\gamma-1} \int_{0}^{\infty} dt\; t^{-\gamma-1} J_{0}(2[xt]^{-\gamma/2})g(t^{-\gamma}) .
\end{equation}
Here $ J_{0}(z) $ is the standard Bessel function of order zero.
The inversion of this Hankel transform is
\begin{equation}\label{h_inverse}
   t^{-\gamma/4}g(t^{-\gamma}) = \int_{0}^{\infty} dx\; J_{0}(2[tx]^{-\gamma/2}) h(x) .
\end{equation}
\end{lemma}
\begin{proof}
As a consequence of \eqref{G_defn} we have
\begin{equation}\label{H-G-relation}
   \mathcal{M}[h(x);s] = \frac{\Gamma \left(\frac{\displaystyle 1+\gamma-s}{\displaystyle \gamma}\right)}{\Gamma\left(\frac{\displaystyle s-1}{\displaystyle \gamma}\right)} \mathcal{M}\left[g(x);\frac{s-1}{\gamma}\right] .
\end{equation}
Given that the inverse Mellin transform of the ratio on the right-hand side of the above equation is available, see Eq. (1), \S 5.31 of Luke \cite{Luke_1975},
we have
\begin{equation*}
	\mathcal{M}^{-1}\left[ \frac{\Gamma \left(\frac{\displaystyle 1+\gamma-s}{\displaystyle \gamma}\right)}{\Gamma\left(\frac{\displaystyle s-1}{\displaystyle \gamma}\right)};x \right]
	= \gamma x^{-\gamma-1} J_{0}(2x^{-\gamma/2}) .
\end{equation*}
The stated results then follow from general identities for the Mellin transform of a multiplicative convolution and a linear transformation of transformed variable.
\end{proof}

\begin{remark}
A check can be made in the special case of $ \alpha = 0 $. Here $ F_j =1 $ and $ g(x) = e^{-x} $.
One finds that \eqref{h_generic} reduces to
\begin{equation*}
   h(x) = 2\gamma x^{-\gamma-1}\int_{0}^{\infty} dv\; v e^{-v^2}J_{0}(2x^{-\gamma/2}v) ,
\end{equation*}
which is a known integral, see \cite[pg. 716, 6.631, \#(1)]{GR_1980}, and that this evaluates to $ \frac{1}{2}\exp(-x^{-\gamma}) $.
Thus we recover $ h_0(x) $ as given by \eqref{h_zero}.
\end{remark}

\begin{remark}
In the case $ \alpha = 1 $ we find
\begin{equation*}
  F_j = \Gamma(1-\gamma)\frac{\Gamma(1+j\gamma)}{\Gamma(1+(j-1)\gamma)}, \quad 
        \prod_{j=0}^{l}F_j = \left( \Gamma(1-\gamma) \right)^{l}\Gamma(1+l\gamma) .
\end{equation*}
Consequently we can recover a well-known result.
\begin{proposition}[\cite{Sch_1987}, \cite{Uchaikin+Zolotarev_1999}]
The generating function $g$ for $ \alpha = 1 $ is a Wright Bessel function of the form
\begin{equation}\label{g_unity}
   g(x) = \phi(\gamma,1,-x/\Gamma(1-\gamma)) ,
\end{equation}
in the notation of \cite[pg. 241]{KL_2019}.
Other notations express this as, see \cite{Wri_1935},
\begin{equation*}
  g(x) = J^{\gamma}_{0}(x/\Gamma(1-\gamma)) ,
\end{equation*}
or 
\begin{equation*}
  g(x) = W(\gamma,1;-x/\Gamma(1-\gamma)) .
\end{equation*}
\end{proposition}
Using the known Mellin transform of the Wright Bessel function we can compute that
\begin{equation*}
   G(t) = \frac{1}{\Gamma\left(\frac{s-1}{\gamma}\right)} \mathcal{M}\left[g(x);\frac{s-1}{\gamma}\right] = \frac{\left(\Gamma(1-\gamma)\right)^{(s-1)/\gamma}}{\Gamma(2-s)} .
\end{equation*}
This agrees with the evaluation of $ G(t) $ using $ H_{1}(s) $ from \eqref{H_unity}.
Combining \eqref{g_unity} and \eqref{h_generic} with \eqref{hFox_unity}	we are led to the following identity
\begin{multline}\label{wright-bessel_fox-h_identity}
   \gamma^2 x^{-\gamma-1} \int_{0}^{\infty} dt\; t^{-\gamma-1} J_{0}(2[xt]^{-\gamma/2})\phi(\gamma,1,-t^{-\gamma}/\Gamma(1-\gamma)) 
   \\
   =
   \gamma^{-1}\left( \Gamma(1-\gamma) \right)^{-1/\gamma} 
   {\rm H}^{0,1}_{1,1} \left( \frac{x}{\left( \Gamma(1-\gamma) \right)^{1/\gamma}};
   \begin{array}{c} (1-\gamma^{-1},\gamma^{-1}) \\ (0,1) \end{array} \right) . 
\end{multline}
It remains to directly verify this.
\end{remark}

\begin{lemma}
Let $ 0 < \alpha < 1 $.
The finite product $ G_l $, defined by \eqref{G-values}, can be written as an analytic function of complex $ \Re(l) > 1+\gamma^{-1} $ in an infinite product form
\begin{equation}\label{Gl}
   G_{l} = (1-\alpha)^{\gamma l} \prod_{j=1}^{\infty} \frac{F_{j+l}}{F_{j}} ,
\end{equation}
and its interpolating function is 
\begin{equation}\label{Ginterpolate}
   G(-t) = (1-\alpha)^{-\gamma t} \prod_{j=1}^{\infty} \frac{F_{j-t}}{F_{j}} .
\end{equation}
A useful relation, to be used later, is
\begin{equation}\label{Fprod_Id}
   F_{-t} \prod_{j=1}^{\infty} \frac{F_{j-t}}{F_{j}} = F_{\infty} \prod_{j=0}^{\infty} \frac{F_{j-t}}{F_{j}} .
\end{equation}
\end{lemma}
\begin{proof}
To show convergence of the product, \cite[\S 29, Satz 5, s. 228]{Kno_1964}, we compute bounds for
$ \left| \frac{F_{j+l}}{F_{j}}-1 \right| $.
The factors can be written (see \eqref{F_integral})
\begin{equation*}
F_{j} = (1-\alpha)^{-\gamma} - \alpha\gamma \int_{0}^{1} du\; u^{j\gamma}(1-\alpha u)^{-\gamma-1} .
\end{equation*}
Then we note
\begin{align*}
\left| \frac{F_{j+l}}{F_{j}}-1 \right| 
& = \frac{\alpha\gamma(1-\alpha)^{\gamma}\int_{0}^{1} du\;(1-\alpha u)^{-\gamma-1}[u^{j\gamma}-u^{(j+l)\gamma}]}
{\left| 1-\alpha\gamma(1-\alpha)^{\gamma}\int_{0}^{1} du\;(1-\alpha)^{-\gamma-1}u^{j\gamma} \right|} ,
\\
& \leq \frac{\alpha\gamma(1-\alpha)^{-1}\int_{0}^{1} du\;[u^{j\gamma}-u^{(j+l)\gamma}]}
{\left| 1-\alpha\gamma(1-\alpha)^{-1}\int_{0}^{1} du\;u^{j\gamma} \right|} ,
\\
& = \frac{\alpha\gamma^2 l}{\left[ 1+(j+l)\gamma \right]\left[ 1-\alpha(1+\gamma)+j\gamma(1-\alpha) \right]}  
\\
& = \text{${\rm O}(j^{-2})$ as $ j \to \infty $}, 
\end{align*}
uniformly for $ l, \alpha, \gamma $ under the conditions stated.	

We recall $ \lim_{j \to \infty}F_{j} = F_{\infty}=(1-\alpha)^{-\gamma} $ is well defined. 
The situation we have is that we take $ N\to \infty $ for fixed $ l $ so
\begin{align*}
   \prod_{j=1}^{N} \frac{F_{j+l}}{F_{j}} 
  & = \frac{F_{N+1}\ldots F_{N+l}}{F_{1}\ldots F_{l}}, \quad \text{as $N\geq l+1$} ,
\\
  & \mathop{\to}\limits_{N \to \infty} \frac{1}{F_{1}\ldots F_{l}} \lim_{N \to \infty}F_{N+1}\ldots F_{N+l} ,
\\
  & = F_{\infty}^l G_{l} ,
\end{align*}
and \eqref{Gl} follows.
Now set $ l \mapsto l-1 $ or $ t \mapsto t+1 $ in \eqref{Ginterpolate}
\begin{align*}
\prod_{j=1}^{N} \frac{F_{j-1+l}}{F_{j}} 
  & = \frac{F_{N+1}\ldots F_{N+l-1}}{F_{1}\ldots F_{l-1}}, \quad \text{as $N\geq l$} ,
\\
  & \mathop{\to}\limits_{N \to \infty} \frac{1}{F_{1}\ldots F_{l-1}} \lim_{N \to \infty}F_{N+1}\ldots F_{N+l-1} ,
\\
  & = F_{\infty}^{l-1} G_{l-1} ,
\\
\prod_{j=1}^{\infty} \frac{F_{j-1+l}}{F_{j}} 
  & = \frac{F_{l}}{F_{\infty}} \prod_{j=1}^{\infty}\frac{F_{j+l}}{F_{j}} .
\end{align*}
Thus \eqref{Fprod_Id} also follows.
\end{proof}

In summary we have our penultimate result where the generating function and density are shown to possess integral representations.
\begin{proposition}\label{MB_representations}
Let $ 0 \leq \alpha < 1-\delta $, $ \delta>0 $.
The generating function $ g(x;\gamma,\alpha) $ has the Mellin-Barnes integral representation $ 0 < c < 1+\gamma^{-1} $
\begin{equation}\label{MBIR_g}
    g(x) = \frac{1}{2\pi i}\int_{c-i\infty}^{c+i\infty} dt\; x^{-t}\Gamma(t)(1-\alpha)^{-\gamma t}
    \prod_{j=1}^{\infty}\frac{{}_2F_1(\gamma,(j-t)\gamma;1+(j-t)\gamma;\alpha)}{{}_2F_1(\gamma,j\gamma;1+j\gamma;\alpha)} ,
\end{equation}
and consequently the density $ h(x;\gamma,\alpha) $ has a similar representation $ c < 1 $
\begin{equation}\label{MBIR_h}
    h(x) = \frac{\gamma}{x}\int_{c-i\infty}^{c+i\infty} \frac{dt}{2\pi i} x^{-\gamma t}\Gamma(1-t)(1-\alpha)^{-\gamma t}
    \prod_{j=1}^{\infty}\frac{{}_2F_1(\gamma,(j-t)\gamma;1+(j-t)\gamma;\alpha)}{{}_2F_1(\gamma,j\gamma;1+j\gamma;\alpha)}  .
\end{equation}
The Mellin transform of the density $ H(s;\gamma,\alpha) $ is
\begin{equation}\label{MBIR_H}
	H(s) = (1-\alpha)^{1-s} \Gamma\left( \frac{1+\gamma-s}{\gamma} \right)
		\prod_{j=1}^{\infty}\frac{{}_2F_1(\gamma,1-s+j\gamma;2-s+j\gamma;\alpha)}{{}_2F_1(\gamma,j\gamma;1+j\gamma;\alpha)}  .
\end{equation}
\end{proposition}
\begin{proof}
The exponential generating function sum can be written as a Mellin-Barnes inversion integral of \eqref{ramanujan}
\begin{equation}
    g(x) = \frac{1}{2\pi i}\int_{c-i\infty}^{c+i\infty} dt\; x^{-t}\Gamma(t) G(-t) ,
\end{equation}
given convergence of the integrand as $ \Im(t) \to \pm \infty $ on the vertical contour $ \Re(t) = c $. 
The second formula follows from the first under the absolute convergence of both integrals and the validity of the integral
\begin{equation*}
   \int^{\infty}_{0} du\; u^{1-2t}J_{0}(zu) = 2^{1-2t}z^{2t-2}\frac{\Gamma(1-t)}{\Gamma(t)} ,
\end{equation*}
when $ \frac{3}{4} < \Re(t) < 1 $.
\end{proof}
\begin{proof}[Second Proof]
There is a simple and direct way to verify \eqref{MBIR_h}.
For generic $ \gamma \notin \Q $ the full set of singularities of the integrand in \eqref{MBIR_h} are:
a sequence of simple poles due to $ \Gamma(1-t) $ at $ t_k=k $, $ k \in \N $, 
and a double sequence of simple poles due to the denominator factor in $ {}_2F_1(\gamma,(j-t)\gamma;1+(j-t)\gamma;\alpha) $ at $ t_{j,k}=j+\gamma^{-1}k $
for $ j,k \in \N $. Thus we are able to displace the contour to the left so that $ c<0 $, which will be necessary for our derivation.
Now we substitute \eqref{MBIR_h} into the right-hand side of \eqref{integral_eqn}
\begin{equation*}
  {\rm RHS} = \frac{\gamma^2}{2\pi ix} \int_{c-i\infty}^{c+i\infty}dt\; \Gamma(1-t)(1-\alpha)^{-\gamma t}
   \prod_{j=1}^{\infty}\frac{F_{j-t}}{F_{j}} \int_{0}^{x}du\; (x-\alpha u)^{-\gamma} u^{-1-\gamma t} ,
\end{equation*}
where we have abbreviated the parameter dependence of the hypergeometric functions and employed the absolute and uniform convergence of the
integrals for $ \Re(t)<-1 $ to exchange the order.
The $u$-integration evaluates to
\begin{equation*}
   \frac{x^{-\gamma -\gamma t}}{-\gamma t}{}_2F_1(\gamma,-t\gamma;1-t\gamma;\alpha)
    = \frac{x^{-\gamma -\gamma t}}{-\gamma t}F_{-t} .
\end{equation*}
Simplifying the right-hand side using \eqref{Fprod_Id} we find
\begin{equation*}
  {\rm RHS} = \frac{\gamma}{2\pi ix} \int_{c-i\infty}^{c+i\infty}dt\; x^{-\gamma(t+1)}\Gamma(-t)(1-\alpha)^{-\gamma-\gamma t}
  \prod_{j=0}^{\infty}\frac{F_{j-t}}{F_{j}} ,
\end{equation*}
and making a translation of the integration variable $ v=t+1 $ we see
\begin{equation*}
  {\rm RHS} = \frac{\gamma}{2\pi ix} \int_{c+1-i\infty}^{c+1+i\infty}dv\; x^{-\gamma v}\Gamma(1-v)(1-\alpha)^{-\gamma v} \prod_{j=0}^{\infty}\frac{F_{j+1-v}}{F_{j+1}} .
\end{equation*}
Because $ c+1<1 $ we can shift the contour one unit to the left thus restoring the original one, and the product index can be relabelled, 
we conclude that the resulting integral is now identical to \eqref{MBIR_h}.
\end{proof}

\begin{remark}\label{discontinuity}
The observations on the large $N$ behaviour of the products in Prop. \ref{large-k-asymptotics} and the integral representations given above 
show that the previous Proposition does not hold in the neighbourhood of $ \alpha=1 $. 
When $ \alpha=1 $ the finite product has the following evaluation
\begin{align}
    \prod_{j=1}^{L} \frac{{}_2F_1(\gamma,(j-t)\gamma;1+(j-t)\gamma;1)}{{}_2F_1(\gamma,j\gamma;1+j\gamma;1)}
    & = \frac{1}{\Gamma(1-t\gamma)}\frac{\Gamma(L\gamma+1-t\gamma)}{\Gamma(L\gamma+1)} ,
\nonumber 
    \\
    & \mathop{\sim}\limits_{L \to \infty} \frac{1}{\Gamma(1-t\gamma)}\left( L\gamma+\frac{1-t\gamma}{2} \right)^{-t\gamma} ,  
\label{limit}
\end{align}
and so the limit vanishes.
Furthermore we know from \eqref{H_unity} that the inverse Mellin transform shows, for $ c < 1+\gamma $,
\begin{equation}\label{h_unity_2nd}
   h_{1}(x) = \frac{1}{2\pi i}\int_{c-i\infty}^{c+i\infty}ds\; x^{-s}
              \left( \Gamma(1-\gamma) \right)^{(s-1)/\gamma} \frac{\Gamma\left(\frac{\displaystyle 1+\gamma-s}{\displaystyle \gamma}\right)}{\Gamma(2-s)} .
\end{equation}
Clearly there is a transition regime from $ \alpha < 1 $ to $ \alpha=1 $ where the limiting behaviour as $ L \to \infty $
crosses over from one to the other - a so-called "double scaling limit".
We can explicitly demonstrate this with the following matching procedure: using \eqref{limit} the relevant part of the integrand of \eqref{MBIR_h} is
\begin{equation}
   (1-\alpha)^{-\gamma t}\prod_{j=1}^{L} \frac{F_{j-t}}{F_{j}} 
   \mathop{\sim}\limits_{L \to \infty \atop \alpha \uparrow 1} \frac{\left[ \gamma(1-\alpha)L \right]^{-\gamma t}}{\Gamma(1-t\gamma)} .
\end{equation}
Now if we change variables $ t=(s-1)/\gamma $ the integrand of \eqref{MBIR_h} will match that of \eqref{h_unity_2nd} if we simultaneously take 
$ L \to \infty $ with $ \alpha \to 1^{\uparrow} $ such that the product 
\begin{equation*}
   (1-\alpha)L = \frac{1}{\gamma\left( \Gamma(1-\gamma) \right)^{1/\gamma}} ,
\end{equation*} 
is fixed.
\end{remark}

\subsection{The large $ \gamma $ case}
A particularly instructive case of our foregoing analysis is when $ \gamma $ is real and large whilst $ 0 < \alpha < 1 $.
In this way we have a simple example where $ \alpha $ is explicitly exhibited.   
It is apparent that in the opposite case of $ \gamma \to 0 $ only the trivial solution of $ h(x) $ with Dirac-atomic measure at the origin is admissible.
In this regime the leading order behaviour of the hypergeometric function is \cite{Par_2013a}, \cite{CSP_2017} is
\begin{equation}
   {}_2F_1(a+\epsilon \lambda, b; c+\lambda;z) \sim \left( 1-\epsilon z \right)^{-b}, \quad |\lambda| \to \infty ,
\end{equation}
excluding the region around the singular point $ z_S = 1/\epsilon $. 
In our application we require the connection formula
\begin{equation}
      {}_2F_1(\gamma,(k-t)\gamma;1+(k-t)\gamma;\alpha) = (1-\alpha)^{-\gamma}{}_2F_1(\gamma,1;1+(k-t)\gamma;-\tfrac{\alpha}{1-\alpha}) ,
\end{equation}
with $ \lambda=(k-t)\gamma$ and $ 0 < \epsilon= 1/(k-t) < 1 $ so our singular point is
\begin{equation}
   \alpha_S = 1 + 1/(k-1-t) > 1,
\end{equation}
which is excluded from our consideration. Using the definition \ref{F-defn} we have
\begin{equation}
  F_j \mathop{\sim}\limits_{\gamma \to \infty} (1-\alpha)^{-\gamma}\left[ 1+\frac{1}{j}\frac{\alpha}{1-\alpha} \right]^{-1}, \quad j\geq 1 .
\label{Fj_large_gamma}
\end{equation}
However in the works cited above, and other literature treating the asymptotic expansions of this hypergeometric function,
no bounds or estimates are given for the error term in \eqref{Fj_large_gamma}.
We address this gap with the following lemma.
\begin{lemma}\label{large-gamma_error}
Let $ 0\leq \alpha < 1 $, $ j \in \N $. Then for $ \gamma \in [0,\infty) $ the hypergeometric function $ F_j $ possesses the asymptotic expansion
as $ \gamma \to \infty $
\begin{equation}
	F_j = (1-\alpha)^{-\gamma} \left\{ \left[ 1+\frac{1}{j}\frac{\alpha}{1-\alpha} \right]^{-1} + {\rm O}(\gamma^{-1}) \right\} .
\label{Fj_error}
\end{equation}
\end{lemma}
\begin{proof}
Using the Apostol-Lerch integral representation of our hypergeometric function \eqref{apostol-lerch}, \eqref{Fj_apostol-lerch},
or making the change of variable $u = e^{-x}$ in Euler’s integral formula for the hypergeometric function, we observe
\begin{equation}
	F_j = j\gamma \int_{0}^{\infty}dx\; e^{-j\gamma x} \left( 1-\alpha e^{-x} \right)^{-\gamma} .
\end{equation}
Defining $ e^{\phi(x)} := e^{-j x} \left( 1-\alpha e^{-x} \right)^{-1} $ we note that $ e^{\phi(x)} $ is real and monotone, decreasing on $ [0,\infty) $
so that $ \phi(x) $ has no stationary points on this interval. The only one occurs at $ x_0 = -\log(\frac{j}{\alpha(j-1)})< 0 $. 
Integrating by parts we find
\begin{equation}
	F_j =  (1-\alpha)^{-\gamma} \left[ 1+\frac{1}{j}\frac{\alpha}{1-\alpha} \right]^{-1}
	 + j\int_{0}^{\infty}dx\; e^{\gamma\phi(x)} \frac{\phi^{''}(x)}{\left( \phi^{'}(x)\right)^2} ,
\label{Fj_IbyParts}
\end{equation}
along with 
\begin{equation*}
	\frac{\phi^{''}(x)}{\left( \phi^{'}(x)\right)^2} = \frac{\alpha e^{-x}}{\left[ j-(j-1)\alpha e^{-x} \right]^2} .
\end{equation*}
Employing the bounds
\begin{equation*}
	e^{\phi(x)} \leq (1-\alpha)^{-1} e^{-j x}, \qquad
	\frac{\phi^{''}(x)}{\left( \phi^{'}(x)\right)^2} \leq \alpha e^{-x}\left[ j-(j-1)\alpha \right]^{-2} ,	
\end{equation*}
in the second term of \eqref{Fj_IbyParts}, we deduce that this error term is bounded by
\begin{equation*}
	\frac{\alpha j}{\left[ j-(j-1)\alpha \right]^{2}} \frac{(1-\alpha)^{-\gamma}}{1+j\gamma} .
\end{equation*}
Thus \eqref{Fj_error} follows. 
\end{proof} 

With this knowledge one can easy verify the following results.
\begin{proposition}\label{large-gamma}
For $ 0 \leq \alpha < 1 $ and $ l\in \N $ the hypergeometric product $ G_l $ has the asymptotic expansion as $ \gamma \to \infty $
\begin{equation}\label{P_large-gamma}
   G_l = (1-\alpha)^{l\gamma} \left\{ \frac{\Gamma(l+\frac{1}{1-\alpha})}{l!\Gamma(\frac{1}{1-\alpha})} + {\rm O}(\gamma^{-1}) \right\} .
\end{equation}
Consequently the generating function is given by the confluent hypergeometric function at leading order in the asymptotic expansion
\begin{equation}\label{GF_large-gamma}
    g(x) = {}_1F_1(\tfrac{1}{1-\alpha};1;-(1-\alpha)^{\gamma}x) + e^{-(1-\alpha)^{\gamma}x}{\rm O}(\gamma^{-1}) .
\end{equation}
This hypergeometric function is an entire function of $ x $ and $ \frac{1}{1-\alpha} $ with exponential decay as $ x \to \infty $.
The Mellin transform of the generating function is
\begin{equation}
   \mathcal{M}[g(x);t] = (1-\alpha)^{-\gamma t}\Gamma(t)
   \left\{  \frac{\Gamma(\frac{1}{1-\alpha}-t)}{\Gamma(1-t)\Gamma(\frac{1}{1-\alpha})} + {\rm O}(\gamma^{-1}) \right\},    
\end{equation}
for $ 0 < \Re(t) < \frac{1}{1-\alpha} $, using Oberhettinger's tables \cite[13.49, pg. 145]{Obe_1974}.
Using the Master Theorem \eqref{ramanujan} we obtain the interpolating function
\begin{equation}\label{G_large-gamma}
   G(t) = (1-\alpha)^{\gamma t} \left\{ \frac{\Gamma(\frac{1}{1-\alpha}+t)}{\Gamma(1+t)\Gamma(\frac{1}{1-\alpha})} + {\rm O}(\gamma^{-1}) \right\} .
\end{equation}
\end{proposition}
\begin{proof}
The hypergeometric product \eqref{P_large-gamma} follows directly from \eqref{Fj_error}
and the generating function result \eqref{GF_large-gamma} includes the exponentially suppressing factor $ (1-\alpha)^{l\gamma} $ in the error term.

One can also solve the functional equation \eqref{MT_functional} directly.
Expanding the coefficient of this equation to sufficient order for large $ \gamma $ we find
\begin{multline}
  \frac{\gamma}{1+\gamma-s} {}_2F_1(\gamma,1+\gamma-s;2+\gamma-s;\alpha)
\\
  = \frac{\gamma}{1+\gamma-s}(1-\alpha)^{1-\gamma} {}_2F_1(2-s,1;2+\gamma-s;\alpha) ,
\\
  = (1-\alpha)^{1-\gamma} \left\{ 1+\frac{\alpha}{\gamma}+\frac{(\alpha-1)(1-s)}{\gamma} + \rm{O}(\gamma^{-2}) \right\} , 
\\
  = (1-\alpha)^{-\gamma} \left[ \frac{1-\alpha/\gamma}{1-\alpha}+\frac{1-s}{\gamma} \right]^{-1}\left\{ 1+\rm{O}(\gamma^{-2}) \right\} .
\end{multline}
The resulting functional equation is 
\begin{equation}
   \frac{H(s)}{H(s-\gamma)} = \frac{(1-\alpha)^{-\gamma}}{\frac{1-\alpha/\gamma}{1-\alpha}+\frac{1-s}{\gamma}} ,
\end{equation}
and solving this in the same manner as in the proof of \eqref{H_unity_Prop} and appealing to the bound \eqref{H_bound} for $ \Im(s) \to \pm\infty $ we deduce
\begin{equation}\label{H_large-gamma}
   H(s) = (1-\alpha)^{1-s} \frac{\Gamma(\frac{1-\alpha/\gamma}{1-\alpha}+\frac{1-s}{\gamma})}{\Gamma(\frac{1-\alpha/\gamma}{1-\alpha})} .
\end{equation}
Recalling \eqref{H-G-relation} this gives us \eqref{G_large-gamma} after discarding the small term $\alpha/\gamma$.
\end{proof}

\begin{remark}
Inverting the Mellin transform \eqref{H_large-gamma} gives an elementary function form for the density
\begin{equation}\label{h_large-gamma}
   h_{\gamma}(x) = \frac{\gamma(1-\alpha)^{-\frac{\gamma-\alpha}{1-\alpha}}}{\Gamma\left(\frac{1-\alpha/\gamma}{1-\alpha}\right)}
   x^{-1-\frac{\gamma-\alpha}{1-\alpha}} \exp\left[ -(1-\alpha)^{-\gamma}x^{-\gamma} \right] .
\end{equation}
Clearly this result is invalid for $ \gamma \leq \alpha $ which, of course, is expressly ruled out by virtue of our assumptions.
\end{remark}

The density exhibits a tail as $ x\to \infty $ with algebraic decay, having the leading form $ x^{-1-\gamma} $ for all $ \gamma>0 $ and $ 0<\alpha\leq 1 $.
For non-rational $ \gamma $ there is a cascade of lower order terms with algebraic decay $ x^{-l-m\gamma} $ for $ l,m \in \N $.
This is illustrated by formula \eqref{h_unity_expand} in the case $ \alpha=1 $ and from the general formula \eqref{MBIR_h} for $ 0<\alpha< 1 $.
This latter formula can be evaluated as a sum over the algebraically decaying terms 
by enclosing the simple poles of the gamma function at $ t\in\N $ and the numerator hypergeometric function factors at $ t\in\N+\gamma\N $ to the right of the contour, 
similar to \eqref{h_unity_expand}
\begin{equation}
	h(x;\gamma,\alpha) \mathop{\sim}\limits_{x \to \infty} 
	\sum_{j\geq 1} a_j x^{-1-\gamma j} + \sum_{l,m\geq 1} b_{l,m} x^{-1-l-\gamma m} ,
\end{equation}
with computable coefficients $ \{a_j\}_{j\geq 1} $, $  \{b_{l,m}\}_{l,m\geq 1} $.
The coefficients $ \{a_j\}_{j\geq 1} $, $  \{b_{l,m}\}_{l,m\geq 1} $ are respectively the residues of the integrand at $ t_j=j, j\in \N $ due to the gamma function
and residues of the integrand at $ t_{l,m}=l+\gamma m, l,m\in \N $ due to the numerator hypergeometric functions.

The behaviour as $ x\to 0^{+}, \arg(x)=0 $ is quite different, however.
The density has an algebraic divergence here, but it is totally overwhelmed by an exponential suppression. 
This is simply illustrated by the large-$\gamma$ formula \eqref{h_large-gamma} and the $ \alpha=1 $ special cases \eqref{alpha-unity_gamma-rational}.
The asymptotic behaviour of the Fox H-function as $ x\to 0 $ with $ \arg(x)=0 $ is known,
see \cite{MSH_2010} pg. 2 Definition 1.1 and Theorem 1.1, Case 4, Eq. (1.109) in (ii) of Theorem 1.3.
However a direct calculation of the saddle point expansion applied to \eqref{ILT_unity} gives this answer plus the constant pre-factors.
Thus for $ 0<\gamma<1 $ and $ \alpha=1 $ we verify the well-known asymptotic estimate, see (3.19)-(3.21) of \cite{Sch_1987},
\begin{multline}
	h_1(x) \mathop{\sim}\limits_{x \to 0^{+}}
	(2\pi(1-\gamma))^{-1/2}\gamma^{1/2(1-\gamma)}\left( \Gamma(1-\gamma) \right)^{1/2(1-\gamma)}x^{-(2-\gamma)/2(1-\gamma)}
\\   \times
	\exp\left[ -(1/\gamma-1)\gamma^{1/(1-\gamma)}\left( \Gamma(1-\gamma) \right)^{1/(1-\gamma)}x^{-\gamma/(1-\gamma)} \right] .
\end{multline}

An outstanding task is to rigorously establish the analogous leading order asymptotics for $ x\to 0$ and general $ 0\leq \alpha <1 $ and $ \gamma>0 $,
which we do not undertake here as this would be beyond the scope of our work.
We conclude by conjecturing that the leading order small-$x$ form, based upon the evidence described above, is
\begin{equation}
	h(x;\gamma,\alpha) \mathop{\sim}\limits_{x \to 0^{+}} C(\gamma,\alpha)x^{-1-d(\gamma,\alpha)} \exp\left[ -p(\gamma,\alpha)x^{-r(\gamma,\alpha)} \right] ,
\label{h_small-x}
\end{equation}
for real, positive $ d,p,r $ and the implied norming constant $C$.
Since the submitted version of our work, Simon \cite{Sim_2023} has recently communicated some developments which go some way to answering the question posed here, 
in particular his Theorem 2.

\subsection{Continued Fraction Forms}

Our final representation of $ g(x) $ is as a continued fraction, 
the utility of which can lie in the fact that they may have better convergence properties than series or products forms. 
In our case the continued fraction form is best expressed in terms of a related function rather than $ g $ itself.
\begin{definition}
Let $ L(z;\gamma,\alpha) $ denote the Laplace transform of the generating function $ g(x) $, for $ \Re(z) > 0 $,
\begin{equation}\label{L_defn}
   L(z) := \mathcal{L}\left[ g(x);z \right] = \int_{0}^{\infty} dx\; e^{-zx} g(x) .
\end{equation}
Thus we have the convergent expansion due to \eqref{G_asymptote} for $ |z|>(1-\alpha)^{\gamma} $
\begin{equation}\label{L_expansion}
   L(z) = \sum_{l=0}^{\infty} \frac{(-1)^l}{\prod_{j=0}^{l} F_{j}} z^{-l-1} .
\end{equation}
\end{definition}
Our definition is also motivated by a relation to the Stieltjes moment problem. 
Setting $X$ for the positive random variable with density $h(x)$, the power transform $Y = X^{-\gamma}$ has the exponential moment generating function
\begin{equation*} 
	\mathbb{E}[e^{zY}] = 1 + \sum_{k\geq 1}\frac{H_{k}z^{k}}{k!}
\end{equation*} 
and in particular the random variable $Y$ is characterized by its integer moments. The above moment generating function is given by
\begin{equation*}\label{key}
	\mathbb{E}[e^{zY}] = -z^{-1}L(-z^{-1}),
\end{equation*}
as follows from \eqref{L_expansion}.

\begin{proposition}
$ L(z) $ is a special $M$-continued fraction, see \cite[pg. 38, Chap. 1]{CPVWJ_2010}, given by
\begin{equation}\label{L_continued-fraction}
   \frac{1}{L(z)} = z+\vcenter{\hbox{\huge $\mathrm{K}$}}_{m=1}^{\infty} \cfrac{\displaystyle F_{m-1}z}{\displaystyle F_{m}z-1} ,
\end{equation}
where our notation is equivalent to the alternative forms
\begin{equation*}
 b_0 + \vcenter{\hbox{\huge $\mathrm{K}$}}^{\infty}_{j=1}\frac{a_j}{b_j}
  = b_0 + \cfrac{a_1}{b_1+\cfrac{a_2}{b_2+\cfrac{a_3}{b_3+\ldots}}}
  = b_0 + \frac{a_1}{b_1} \genfrac{}{}{0pt}{}{}{+}
          \frac{a_2}{b_2} \genfrac{}{}{0pt}{}{}{+} 
          \frac{a_3}{b_3} \genfrac{}{}{0pt}{}{}{+\ldots} .
\end{equation*}
\end{proposition}
\begin{proof}
Eq. \eqref{L_continued-fraction} is found by adapting a result known to Euler, see \cite[(1.7.2)]{CPVWJ_2010}, for the recasting of a series
\begin{equation}\label{e_series}
   e_{n} = 1+\sum_{k=1}^{n} \prod_{j=1}^{k}\big(-E_{j}\big)z^k ,
\end{equation}
into
\begin{align*}
   e_{n} = \cfrac{1}{1+\vcenter{\hbox{\huge $\mathrm{K}$}}^{n}_{j=1}\cfrac{\displaystyle E_{j}z}{\displaystyle 1-E_{j}z}} .
\end{align*}
It is straight forward to show the partial numerators and denominators $ A_n $ and $ B_n $ of $ e_{n} = \frac{A_n}{B_n} $
satisfy the recurrence relations
\begin{align*}
   A_n & = b_nA_{n-1}+a_nA_{n-2} ,\\
   B_n & = b_nB_{n-1}+a_nB_{n-2} ,
\end{align*}
with $ b_0=0, a_1=1, b_1=1 $ and for $ m \geq 1 $, $ a_{m+1} = E_{m} z $ and $ b_{m+1} = 1-E_{m}z $.
The solutions to this recurrence give $ B_{n}=1 $ and $ A_{n} $ the right-hand side of \eqref{e_series}. 
\end{proof}

\begin{remark}
As a check for $ \alpha=0 $ we recall $ F_j=1 $, $ j \geq 0 $ so that
\begin{equation*}
   \frac{1}{L_{0}(z)} 
   = z + \cfrac{z}{z-1+\cfrac{z}{z-1+\cfrac{z}{\ddots}}}
   = z+R ,
\end{equation*}
however we observe $ R = \dfrac{z}{z-1+R} $ so that $ (R-1)(R+z) = 0 $. 
Choosing $ R=1 $ as the appropriate root we verify that 
\begin{equation}
   L_{0}(z) = \frac{1}{z+1} , 
\end{equation}
which should be compared with the Laplace transform of $ g_{0}(x)=e^{-x} $.
\end{remark}

\begin{remark}
For $ \alpha=1 $ we can record a result which may be known
\begin{equation}
  \frac{1}{L_{1}(z)} 
  = z + \vcenter{\hbox{\huge $\mathrm{K}$}}_{m=1}^{\infty} \cfrac{\frac{\Gamma(1+(m-1)\gamma)}{\Gamma(1+(m-2)\gamma)}\Gamma(1-\gamma)z}
                                                                 {\frac{\Gamma(1+m\gamma)}{\Gamma(1+(m-1)\gamma)}\Gamma(1-\gamma)z-1} .
\end{equation}
On the other hand the Laplace transform of $ g_{1} $ is the Mittag-Leffler function $ E_{\gamma}(\cdot) $, see \cite[pg. 269]{KL_2019} 
\begin{equation}
   L_{1}(z) = \frac{1}{z} E_{\gamma}\big(-\tfrac{1}{\Gamma(1-\gamma)z}\big) .
\end{equation}
This is an entire function of $ z $ of order $ \frac{1}{\Re(\gamma)} $ and type unity.
\end{remark}

\begin{remark}
From the results of Prop. \ref{large-gamma} we can show that for $ \gamma \to \infty $ the leading term of the Laplace transform has the algebraic form
\begin{equation}
   L_{\gamma}(z) = z^{-1}\left( 1+ (1-\alpha)^{\gamma}z^{-1} \right)^{-\frac{1}{1-\alpha}} 
   					+ z^{-1}e^{-(1-\alpha)^{\gamma}z^{-1}}{\rm O}(\gamma^{-1}).
\end{equation}
This can be established from the convergent expansion \eqref{L_expansion} about $ z=\infty $ using \eqref{P_large-gamma} or
computing the Laplace transform of \eqref{GF_large-gamma} with the assistance of \cite[pg. 215, \S2.2 (11)]{EMOT_1954a} assuming $ \Re(z)>0 $.
Therefore we note the M-continued fraction form 
\begin{equation}
  \frac{1}{L_{\gamma}(z)} 
  = z + \vcenter{\hbox{\huge $\mathrm{K}$}}_{m=1}^{\infty} \cfrac{\frac{(m-1)}{(m-1+\frac{\alpha}{1-\alpha})} (1-\alpha)^{-\gamma}z}
    {\frac{m}{(m+\frac{\alpha}{1-\alpha})} (1-\alpha)^{-\gamma}z-1} .
\end{equation}
\end{remark}
	


\section{Numerical Evaluations}\label{numerics}
\setcounter{equation}{0}

Our final results are numerical computations of the density $h(x;\gamma,\alpha)$ to within stated global accuracies using formulae derived in our work in order to 
illustrate some of the foregoing theoretical statements concerning the behaviour of these solutions.
The most robust numerical strategy for computing the density is to use the Mellin-Barnes integral representation \eqref{MBIR_h}
because the integrand is smooth and decays exponentially as $ \Im(t) \to \pm\infty $. 
There are oscillations but they have a period of the same order as the foregoing decay length and therefore are rapidly damped.
Our methods employed 64-bit or {\tt real (kind=8)} precision with the gfortran F90 compiler {\tt GNU Fortran (Ubuntu 7.5.0-3ubuntu1~18.04) 7.5.0} 
and software from the following sources:
\begin{itemize}
\item 
We have employed DQAGI for the Mellin-Barnes integrals and DQAGS for the normalisation and integral equations from the QUADPACK package\cite{PdDUK_1983}.
The absolute and relative error goals were $10^{-6}$ and $10^{-4}$ respectively. A convenient choice for $c$ is half.
\item 
The complex gamma function was computed with the error control ACM algorithm 421 \cite{Kuk_1972} with the error value set at $10^{-8}$.
\item 
The complex Gauss hypergeometric function was computed using algorithm {\tt AEAE\_v1\_0} from the CPC program library \cite{MS_2008}.
An error criterion is computed from certain contiguous relations.
\item 
Extrapolation of the Gaussian hypergeometric function products up to $ 2^{13} $ factors was performed by ACM algorithm 602 HURRY, 
an acceleration algorithm for scalar sequences and series \cite{FDAS_1983}. 
\end{itemize} 
Two criteria were used to quantify the all-up global errors - 
the normalisation of the density and the residual of the solution in satisfying the integral equation \eqref{integral_eqn} at $ x=1 $.
These are tabulated in Tables \ref{norm-error-table} and \ref{residual-error-table} respectively over a range of $ \gamma\in (0,1] $ and $ \alpha\in (0,1) $ values.
Our error settings are relatively weak and no attempt has been made in pushing the methods to their accuracy limits. 

\begin{table}[H]
	\caption{Absolute Normalisation Errors}
	\label{norm-error-table}       
	\centering
	\begin{tabular}{|l|cccccc|}
		\hline
		\backslashbox{$\gamma$}{$\alpha$} & 0.10 & 0.20 & 0.30 & 0.40 & 0.50 & 0.60 \\
		\hline
		0.25 	& 0.318E-05 & 0.254E-04 & 0.633E-04 & 0.333E-06 & 0.888E-05 & 0.111E-05 \\
		0.50	& 0.331E-05	& 0.126E-05	& 0.201E-06	& 0.122E-04	& 0.733E-05	& 0.138E-05	\\
		0.75	& 0.162E-04	& 0.516E-05	& 0.301E-05	& 0.128E-05	& 0.106E-06	& 0.150E-05	\\
		1.00	& 0.160E-04	& 0.389E-04	& 0.664E-05	& 0.380E-06	& 0.260E-05	& 0.157E-04	\\
		\hline
	\end{tabular}
\end{table}

\begin{table}[H]
	\caption{Absolute Residual Errors}
	\label{residual-error-table}       
	\centering
	\begin{tabular}{|l|cccccc|}
		\hline
		\backslashbox{$\gamma$}{$\alpha$} & 0.10 & 0.20 & 0.30 & 0.40 & 0.50 & 0.60 \\
		\hline
		0.25 	& 0.237E-06 & 0.143E-07 & 0.653E-07 & 0.209E-06 & 0.689E-09 & 0.746E-07 \\
		0.50	& 0.126E-05	& 0.517E-08	& 0.122E-06	& 0.176E-06	& 0.677E-07	& 0.341E-07	\\
		0.75	& 0.622E-05	& 0.347E-05	& 0.284E-05	& 0.383E-06	& 0.253E-06	& 0.129E-06	\\
		1.00	& 0.596E-06	& 0.369E-05	& 0.131E-06	& 0.485E-06	& 0.301E-05	& 0.876E-05	\\
		\hline
	\end{tabular}
\end{table}

We present our results in a sequence of Figures 
\ref{fig:gamma=quarter}, 
\ref{fig:gamma=half}, 
\ref{fig:gamma=threequarters}, 
\ref{fig:gamma=unity}
and \ref{fig:gamma=large} plotting the density $h(x;\gamma,\alpha)$ versus $x$ with increasing values of $\gamma$ and within each plot $\alpha$ increases through the range $[0,1)$.


\begin{figure}[H]
	\centering
	\includegraphics[height=80mm]{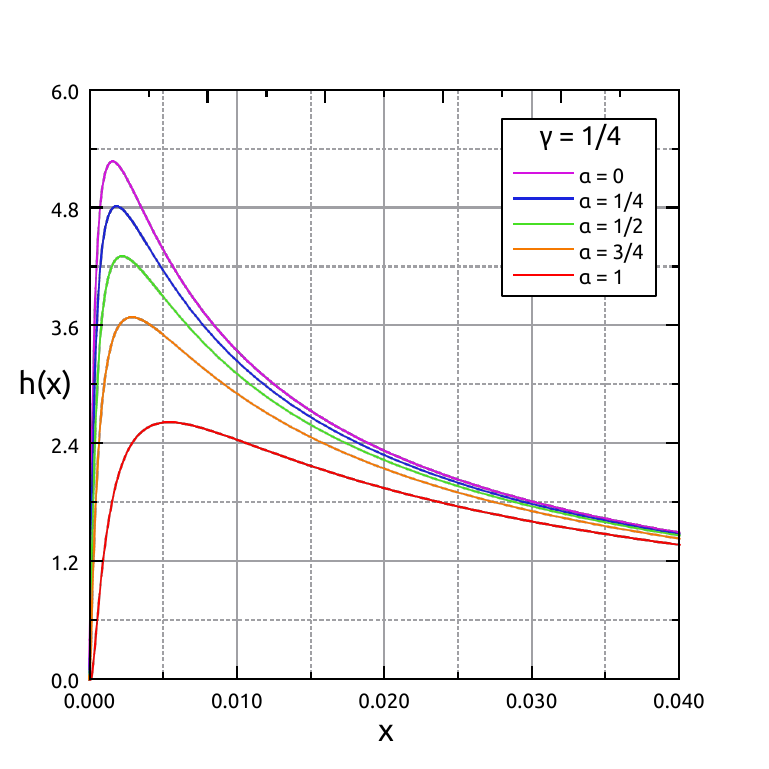}
	\caption{Densities $ h(x;\tfrac{1}{4},\alpha) $ versus $ x $ for $ \alpha = \{ 0, \frac{1}{4}, \frac{1}{2}, \frac{3}{4}, 1 \}$ computed using \eqref{MBIR_h} and \eqref{h_unity}.}
	\label{fig:gamma=quarter}       
\end{figure}

\begin{figure}[H]
	\centering
	\includegraphics[height=80mm]{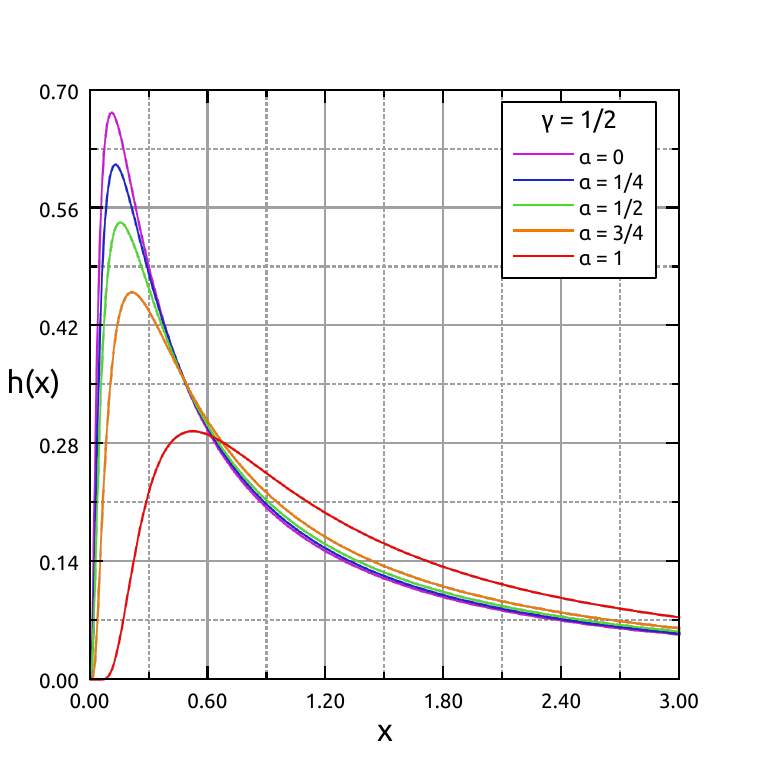}
	\caption{Densities $ h(x;\tfrac{1}{2},\alpha) $ versus $ x $ for $ \alpha = \{ 0, \frac{1}{4}, \frac{1}{2}, \frac{3}{4}, 1 \}$ computed using \eqref{MBIR_h} and \eqref{h_unity}.}
	\label{fig:gamma=half}       
\end{figure}


\begin{figure}[H]
	\centering
	\includegraphics[width=80mm]{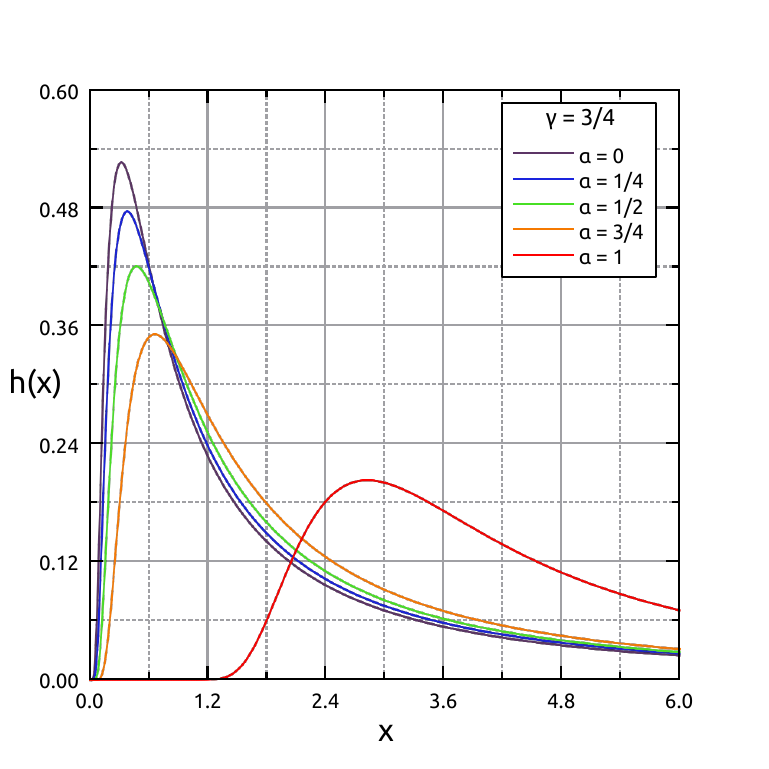}
	\caption{Densities $ h(x;\tfrac{3}{4},\alpha) $ versus $ x $ for $ \alpha = \{ 0, \frac{1}{4}, \frac{1}{2}, \frac{3}{4}, 1 \}$ computed using \eqref{MBIR_h} and \eqref{h_unity}.}
	\label{fig:gamma=threequarters}       
\end{figure}

\begin{figure}[H]
	\centering
	\includegraphics[width=80mm]{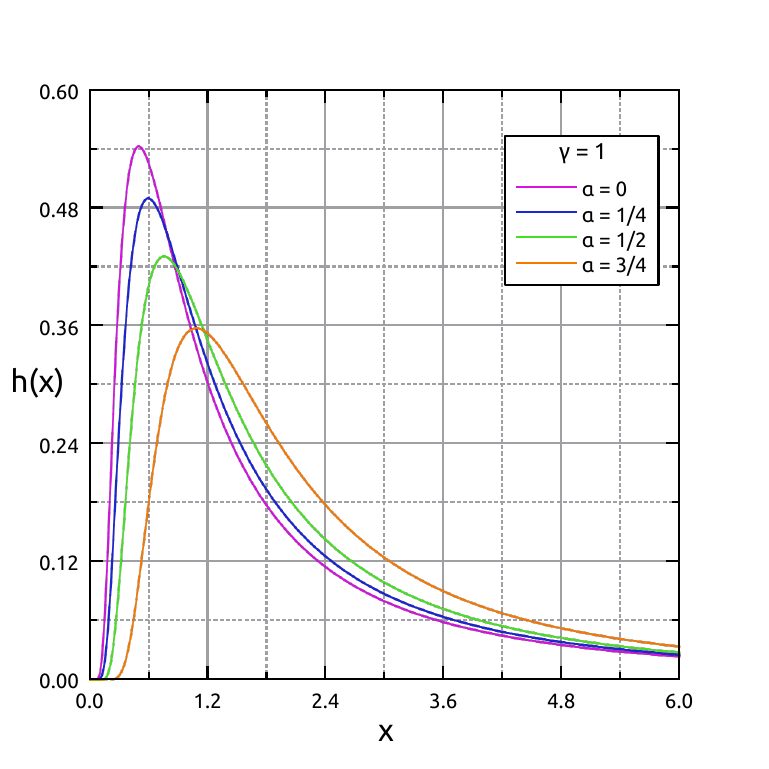}
	\caption{Densities $ h(x;1,\alpha) $ versus $ x $ for $ \alpha = \{ 0, \frac{1}{4}, \frac{1}{2}, \frac{3}{4} \}$ computed using \eqref{MBIR_h}.}
	\label{fig:gamma=unity}       
\end{figure}


\begin{figure}[H]
	\centering
	\includegraphics[width=80mm]{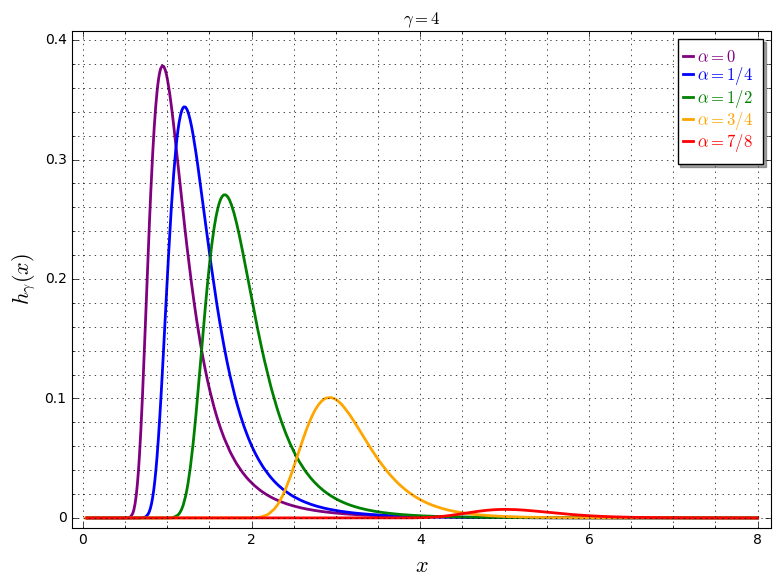}
	\caption{Large-$\gamma$ densities as computed using \eqref{h_large-gamma}. In this case $ \gamma=4 $ and $ \alpha\in \{0,\tfrac{1}{4},\tfrac{1}{2},\tfrac{3}{4},\tfrac{7}{8} \}$.}
	\label{fig:gamma=large}       
\end{figure}



\section{Discussion}\label{discuss}
We pick up here the thread in Section \S \ref{motivation} concerning related works in the probability literature.
Schlather \cite{Sch_2001} treats a different operation from the one considered in \cite{GH_1991}, \cite{HG_1997} or here
in that he takes a $1/p$ root of a sum of $p^{\rm th}$ powers. 
In addition to being an altogether different problem producing a very different solution, he employs a different method. 
For example he finds a different family of limit distributions parametrised by $c$ with $ 0 \leq c \leq \infty $. 
And he discusses domains of attraction, 
but with norming depending on $c$ (unlike in the foregoing works) and of course different limits.

We recently learned \cite{Sim_2023} that the $\alpha$-Sun limit distributions
on $\R^+$ satisfying \eqref{alpha-sun} and represented by \eqref{Mellin-Barnes_h}, $0 < \alpha < 1$, are infinitely divisible. 
This is not unexpected since this property is long-known for the Fr{\'e}chet ($\alpha = 0$) and for the stable ($\alpha = 1$) distributions of normed limits, 
and for the generalised stable laws.
The ingenious proof in \cite{Sim_2023} for the $\alpha$-Sun case uses the relatively recent result of \cite{Bon_2015} that the class (GGC) of generalized gamma
convolutions is stable under independent multiplication. 
A remaining challenge would be to find a more straight-forward proof, see e.g. \cite{GdH_2000}. 
A useful overview of infinite divisibility arguments for classes of density functions up to 2014 is in \cite{Pak_2014}.

For $\alpha = 1$ we have recovered a known identification of $h(x;\gamma,\alpha)$ as a Fox $H$-function, 
and we have mentioned some previous works \cite{Sch_1987}, \cite{Pak_2014}, \cite{JSW_2018} where there is a nice statistical interpretation of this fact.
However the sequence \eqref{alpha-sun} is in the domain of attraction of the normal
distribution for $\alpha = 1$, so there is a ``discontinuity":
\eqref{alpha-sun} interpolates between Fr{\'e}chet and normal, whereas \eqref{integral_eqn} interpolates between Fr{\'e}chet and
this Fox $H$-function. 

Regarding the nature of the generic solution $ h(x;\gamma,\alpha) $ we have found here, 
it appears that this is a novel extension of the hypergeometric function to the best of our knowledge.
Standard extensions through the Mittag-Leffler, Wright-Bessel and Meijer-$G$ functions to the Fox H-functions clearly appear when 
$ \alpha=1 $, but not for $ 0<\alpha<1 $.
The clearest way to distinguish the generic functions found here is to think of the standard hypergeometric or Fox-H function as series
\begin{equation*}
    H\left( \begin{array}{ccc}
               (A_{1},a_{1}) & \ldots & (A_{P},a_{P}) \\
               (B_{1},b_{1}) & \ldots & (B_{Q},b_{Q}) 
            \end{array} ;x \right) = \sum_{n\geq 0} C_n x^n ,
\end{equation*}
and that the coefficients $ C_n $ satisfy a first order difference equation with respect to $n$, 
in unit steps in the forward direction $ n\mapsto n+1 $,
\begin{multline*}
   C_{n+1} = \frac{(A_{1}+a_{1}n) \ldots (A_{P}+a_{P}n)}{(B_{1}+b_{1}n) \ldots (B_{Q}+b_{Q}n)} C_{n},
\\ 
   \left\{ A_{1},a_{1}, \ldots, A_{P},a_{P}, B_{1},b_{1}, \ldots, B_{Q},b_{Q} \right\} \in \C ,
\end{multline*}
i.e. rational coefficients in $n$. 
Such equations can be solved in terms of products of gamma functions or Pochhammer symbols given suitable boundary conditions.
However in our generic case the coefficient of the first order difference equation is itself a hypergeometric function of $n$,
which is the solution of a second-order difference equation with respect to $n$, but as $ n\mapsto n+1/\gamma $.
Alternatively the original coefficients $ C_n $, or effectively $ H(s=1-n\gamma) $, 
would satisfy a non-linear difference equation with rational coefficients in $n $ but with {\it incommensurate steps of $1$ and $\gamma^{-1}$} 
\begin{multline}
  0 = \left[ 1+\gamma-s+\alpha(1-s) \right]\frac{H(s)}{H(s-\gamma)}
\\
      + \alpha(s-2)\frac{H(s-1)}{H(s-1-\gamma)} + (s-\gamma)\frac{H(s+1)}{H(s+1-\gamma)} .
\label{GHF}
\end{multline}
This is an entirely higher level of special function.

It may be of interest to note that another family of extensions to the Gauss hypergeometric functions has been explored recently,
under the monikers of 
{\it extended hypergeometric functions} \cite{CQRZ_1997}, 
{\it extended $\tau$-Gauss hypergeometric functions} \cite{VKA_2001}, \cite{Par_2015},
{\it extended, generalised $\tau$-Gauss hypergeometric functions} \cite{KGS_2016}.
The common theme here is that a pair of Pochhammer symbol ratios in $ C_n $, 
which are equivalent to the beta function $ B(a,b) $ \cite[\S 5.12]{DLMF}, are replaced by
\begin{equation}\label{ext-GaussHF}
   B(a,b;\sigma) = \int_{0}^{1} dt\; t^{a-1}(1-t)^{b-1}\exp\left[ -\frac{\sigma}{t(1-t)}\right] , \Re(\sigma) >0,
\end{equation}
along with
\begin{equation*}
F_{\sigma}(a,b;c;z) = \sum_{n=0}^{\infty} (a)_n \frac{B(b+n,c-b;\sigma)}{B(b,c-b)}\frac{z^n}{n!} ,
\end{equation*}
or ultimately
\begin{equation*}
   B(a,b;\sigma;\alpha,\beta,\rho,\lambda) = 
   \int_{0}^{1} dt\; t^{a-1}(1-t)^{b-1} {}_1F_1\left(\alpha;\beta;-\frac{\sigma}{t^{\rho}(1-t)^{\lambda}}\right) , \Re(\sigma) >0.
\end{equation*}
There is no evidence that the extended beta function \eqref{ext-GaussHF} satisfies a difference equation like that of \eqref{GHF}.

\section*{Acknowledgements}
The authors would like to acknowledge Jesse Goodman for organising the 2020 Taupo Probability Workshop
and the funding support of the University of Auckland, which facilitated this collaboration.
We are also grateful for the insightful comments and observations of Tony Pakes 
on an earlier version of the manuscript. It was also a pleasure to hear of the progress made by Thomas Simon on a number of fronts in his recent preprint \cite{Sim_2023}.
Lastly we wish to express our gratitude to the referee for a careful and thoughtful reading.

\end{document}